\documentclass[moor,sglanonrev]{informs4}
\RequirePackage{tgtermes}
\RequirePackage{newtxtext}
\RequirePackage{newtxmath}
\RequirePackage{bm}

\SingleSpacedXII

\usepackage{algorithm}
\usepackage{algpseudocode}
\usepackage{tikz}
\usepackage[colorlinks=true,breaklinks=true,bookmarks=true,urlcolor=blue,   citecolor=blue,linkcolor=blue,bookmarksopen=false,draft=false]{hyperref}

\def\EMAIL#1{\href{mailto:#1}{#1}}

\usepackage[sort&compress]{natbib}
 \bibpunct[, ]{[}{]}{,}{n}{}{,}%
 \def\bibfont{\small}%
\EquationsNumberedThrough    

\TheoremsNumberedThrough     
\ECRepeatTheorems  %

\MANUSCRIPTNO{MOOR-0001-2024.00}

\usepackage{mathtools}
\usepackage{graphicx,subcaption}

\usepackage[margin=1in]{geometry} 

\usepackage{centernot}
\usepackage[utf8]{inputenc}
\usepackage{breqn}
\usepackage{color-edits}
\addauthor{sh}{magenta}
\addauthor{rp}{blue}
\usepackage[below]{placeins}

\newcommand\xqed[1]{%
  \leavevmode\unskip\penalty9999 \hbox{}\nobreak\hfill
  \quad\hbox{#1}}
\newcommand\demo{\xqed{$\triangle$}}

\newcommand{\inP}{\xrightarrow{\mathclap{\text{\scriptsize{p}}}}}

\newcommand{\as}{\xrightarrow{\text{\scriptsize{a.s.}}}}

\newcommand{\inD}{\xrightarrow{\text{\scriptsize{d}}}}

\renewcommand{\qed}{\tag*{$\blacksquare$}}

\renewcommand{\mathcal}{\mathscr}

\DeclarePairedDelimiterX{\expectarg}[1]{[}{]}{%
  \ifnum\currentgrouptype=16 \else\begingroup\fi
  \activatebar#1
  \ifnum\currentgrouptype=16 \else\endgroup\fi
}

\newcommand{\innermid}{\nonscript\;\delimsize\vert\nonscript\;}
\newcommand{\activatebar}{%
  \begingroup\lccode`\~=`\|
  \lowercase{\endgroup\let~}\innermid 
  \mathcode`|=\string"8000
}

\newcommand{\rmd}{{\mathrm{d}}}

\renewcommand{\theARTICLETOP}{}
\usepackage{fancyhdr}
\begin{document}



\RUNTITLE{Frank-Wolfe on Probability Spaces}

\TITLE{Deterministic and Stochastic Frank-Wolfe \\ Recursion on Probability Spaces}

\ARTICLEAUTHORS{%
\AUTHOR{Di Yu}
\AFF{Purdue University, Department of Statistics, \EMAIL{yu1128@purdue.edu}}
\AUTHOR{Shane G. Henderson}
\AFF{Cornell University, Operations Research and Information Engineering, \EMAIL{sgh9@cornell.edu}}
\AUTHOR{Raghu Pasupathy}
\AFF{Purdue University, Department of Statistics, \EMAIL{pasupath@purdue.edu}}
} 

\ABSTRACT{Motivated by applications in emergency response and experimental design, we consider smooth stochastic optimization problems over probability measures supported on compact subsets of the Euclidean space. With the \emph{influence function} as the variational object, we construct a deterministic Frank-Wolfe (dFW) recursion for probability spaces. The dFW recursion is made especially possible by a lemma that identifies the solution to the infinite-dimensional Frank-Wolfe sub-problem as a Dirac measure concentrating on the minimum of the influence function at the incumbent iterate. Each iterate in dFW is thus expressed through a ``particle update,'' as a convex combination of the incumbent iterate and a Dirac measure. To address common application contexts that have access only to Monte Carlo observations of the objective and influence function, we construct a stochastic Frank-Wolfe (sFW) variation that generates a random sequence of probability measures constructed using minima of increasingly accurate estimates of the influence function. We demonstrate that sFW's optimality gap sequence exhibits $O(k^{-1})$ iteration complexity almost surely and in expectation for smooth convex objectives, and $O(k^{-1/2})$ (in Frank-Wolfe gap) for smooth non-convex objectives. Furthermore, we show that an easy-to-implement fixed-step, fixed-sample version of (sFW) exhibits exponential convergence to $\varepsilon$-optimality. We end with a central limit theorem on the observed objective values at the sequence of generated random measures. To further intuition, we include several illustrative examples with exact influence function calculations.}





\maketitle


\section{INTRODUCTION.} Consider the ``out-of-hospital cardiac arrest'' (OHCA)~\cite{2018myaetal} emergency response problem where a person experiencing an OHCA event needs immediate medical assistance. Volunteers serve as the ``first responders'' to incoming OHCA events, in addition to the usual ambulance response, thereby elevating survival rates. One wants to identify how volunteers should be concentrated so as to maximize the expected survival probability of the next OHCA patient, while recognizing that incoming OHCA events have random locations. Such information can aid in targeted volunteer recruitment efforts or to provide bounds on survival rates for a given number of volunteers.

Let's pose the OHCA problem more concretely, as in~\cite{2022shaetal}.  Suppose $\mu$ represents the concentration of volunteers, that is, the probability measure that, when scaled by the number of volunteers, gives the measure governing the location of volunteers in a city represented by a compact set $\mathcal{X} \subset \mathbb{R}^2.$ Suppose also that $\eta(\cdot)$ denotes a (spatial) probability measure governing the location of an OHCA event supported on $\mathcal{X}$, and $\beta_0: [0,\infty) \to [0,1]$ is a non-decreasing right-continuous function representing the probability of the OHCA patient dying for a given response time. Finally, let $R_{x,\mu} \in \mathbb{R}^+$ be a $\mu$-dependent random variable corresponding to the first response time to an incident occurring at $x \in \mathcal{X}$. Then, assuming that $\mu$ can be chosen, the OHCA emergency response problem seeks a probability measure $\mu$ supported over $\mathcal{X}$ that minimizes the expected probability of death $J(\mu) := \int_{\mathcal{X}} \eta(\mbox{d}x)\int_0^{\infty} P(\beta_0(R_{x,\mu}) \geq u) \, \mbox{d}u$ of the OHCA patient. 

The OHCA emergency response problem is an instance of the following broader class of optimization over probability spaces that forms the focus of this paper: \begin{alignat}{2}\label{mainopt} \!\min_{\mu}\quad  & J(\mu) \nonumber \\ \mbox{subject to} \quad & \mu \in \mathcal{P}(\mathcal{X}). \tag{$P$} \end{alignat} In problem~\eqref{mainopt}, $\mathcal{X}$ is a compact convex subset of $\mathbb{R}^d$, $\mathcal{P}(\mathcal{X})$ is the probability space on $\mathcal{X}$, that is, the space of non-negative Borel measures $\mu$ supported on $\mathcal{X}$ such that $\mu(\mathcal{X}) = 1$, and $J: \mathcal{P}(\mathcal{X}) \to \mathbb{R}$ is a real-valued function having domain $\mathcal{P}(\mathcal{X})$. The objective $J$ is assumed to be known through a blackbox oracle~\cite{2018nes} that returns $J(\mu) \in \mathbb{R}$ at any requested $\mu \in \mathcal{P}(\mathcal{X})$. Indeed, while~\eqref{mainopt} is a variation on the well-studied problem of optimization over the space of measures~\cite{2002molzuy,2000molzuy,2022achi,2022bchi,2013brepik}, recent applications from high-dimensional statistics, signal processing, and machine learning~\cite{2013canfer,2014canfer,2017boygeorec,2012degam,2013fer,2019chublagly,kent2021frankwolfe} have drawn interest in specialized solution methods for~\eqref{mainopt}, especially in light of strides and interest in primal methods for solving corresponding constrained optimization problems over Euclidean spaces. See also Section~\ref{sec:examples} for further concrete examples.

It is often the case that in practical applications, e.g., the OHCA emergency response problem and various examples in Section~\ref{sec:examples}, no blackbox oracle for $J$ is available. To cover such settings, we also consider a variation of~\eqref{mainopt} where the objective $J$ is known through a stochastic oracle, that is, an oracle capable of providing unbiased Monte Carlo samples of $J$ at a requested $\mu \in \mathcal{P}(\mathcal{X})$. Formally, we write this problem as \begin{alignat}{2}\label{smainopt} \!\min_{\mu} \quad & J(\mu)  = \mathbb{E}[F(\mu,Z)] = \int F(\mu,z) \, P(\mathrm{dz}) \nonumber \\ \mbox{subject to} \quad & \mu \in \mathcal{P}(\mathcal{X}), \tag{$sP$} \end{alignat}
where $Z$ is a random variable having distribution $P$ on a measurable space $(\mathcal{Z},\mathcal{A})$, and the function $F(\cdot,\cdot): \mathcal{P}(\mathcal{X}) \times \mathcal{Z} \to \mathbb{R}$ provides for a stochastic unbiased oracle in that $F(\mu,z)$ can be observed at a requested $\mu \in \mathcal{P}(\mathcal{X}), z \sim Z$ with $\mathbb{E}[F(\mu,Z)] = J(\mu)$ for each $\mu \in \mathcal{P}(\mathcal{X})$. So, while objective $J$ in~\eqref{smainopt} is unobservable, an unbiased estimate of $J$ at any requested $\mu \in \mathcal{P}(\mathcal{X})$ can be constructed by drawing independent and identically distributed samples from the distribution of $Z$. When devising a solution recursion to~\eqref{smainopt}, we will assume access to a first-order stochastic oracle that provides unbiased estimates of a useful mathematical object called the \emph{influence function} --- see Section~\ref{sec:lit} and Section~\ref{sec:preliminaries} for further discussion.

\subsection{Summary and Contribution} Our main contibutions are as follows. 
\begin{enumerate}
\item[(a)] We derive a variational Frank-Wolfe recursion operating directly on probability spaces associated with~\eqref{mainopt}, and using the influence function of the objective $J$ as the first-order variational object. Our treatment stipulates only that $J$ possess a certain weak form of the derivative, and in particular, neither stipulates that $J$ has a ``convex loss of linear functional'' (CLLF) structure (see Section~\ref{sec:lit}), nor that the solution to~\eqref{mainopt} is \emph{sparse}, that is, supported on a countable subset of $\mathbb{R}^d$. The introduction of the influence function as the variational object within a recursion seems to have first appeared in~\cite{2019chublagly}. 

\item[(b)] Analogous to calculations in the CLLF context over the space of signed measures, {the deterministic Frank Wolfe recursion~\eqref{dFW} (introduced in Section~\ref{sec:dfw}; see Algorithm~\ref{alg:DFW})} is shown~(see Lemma~\ref{lem:dFW}) to have a sub-problem with a ``closed-form'' solution. Consistent with the historical viewpoint~\cite{2017boygeorec,pokutta2023frankwolfe,1978dunhar}, this ability to efficiently solve the Frank-Wolfe sub-problems is critical to implementation as a ``particle update'' using Dirac measures (see also~\cite{2019caretal,2019futcui,2016liudil,2022caretal}), and also to constructing key variations such as fully corrective Frank-Wolfe (see Algorithms~\ref{alg:DFW} and~\ref{alg:fullycorrectivestoch}). Our use of the influence function within a Frank-Wolfe recursion presents an interesting contrast to~\cite{kent2021frankwolfe}, where a Wasserstein derivative~\cite{2005ambluisav,2015san} is used within a Frank-Wolfe recursion for optimizing a functional $J$ defined over the (smaller) space $\mathcal{P}_2(\mathbb{R}^d)$ of Borel probability measures equipped with the Wasserstein metric of order 2. Since no closed-form solution is evident, the authors in~\cite{kent2021frankwolfe} present an elaborate algorithm to solve the resulting sub-problems. An explicit relationship can be established between the influence function introduced here and the Wasserstein derivative in~\cite{kent2021frankwolfe}, but we do not go into further detail.    

\item[(c)] We show~(see Theorem~\ref{thm:dFWcomp}) that the iterates resulting from~\eqref{dFW} enjoy $O(k^{-1})$ iteration complexity in functional value under a smoothness assumption on the objective $J$. For stochastic and statistical settings where only unbiased estimates of the influence function are available, we present a stochastic analogue~\eqref{sfw} of~\eqref{dFW}, {introduced in Section~\ref{sec:sfw}; see Algorithm~\ref{alg:fullycorrectivestoch}.} This version solves a sampled version of the Frank-Wolfe sub-problem. Here again, a ``closed-form'' solution to the Frank-Wolfe sub-problem expressed in terms of the minimum of the sampled influence function emerges. Theorem~\ref{thm:sfwcomplexity} identifies a stepsize and sample size relationship to guarantee $O(k^{-1})$ complexity in expectation, and Theorem~\ref{thm:sfwasconsistency} identifies an almost sure convergence rate. Theorem~\ref{thm:sfwclt} identifies the exact asymptotic distribution of the estimated functional values at the~\eqref{sfw} iterates through a central limit theorem.
\item[(d)] In settings where the objective $J$ is nonconvex, Theorem~\ref{thm:sfwcomplexNonConvex} demonstrates that under a certain choice of the step size and sample size, the so-called Frank-Wolfe gap attains the $O(1/\sqrt{T})$ bound that is also achieved in Euclidean settings.
\item[(e)] Given our viewpoint of the influence function as the first-order variational object when solving stochastic optimization problems on probability spaces, we prove a number of optimization-related basic structural results expressed in terms of the influence function. For example,  conditions on optimality (Lemma~\ref{lem:optCondition}), nature of the support of the optimal measure (Lemma~\ref{lem:supportopt}), and continuity of the influence function (Lemma~\ref{lem:infcont}).
\end{enumerate}

\subsection{Paper Organization} The ensuing Section~\ref{sec:lit} provides perspective on the relationship of the existing literature with the current paper especially from the standpoints of the CLLF structure, influence function, and sparsity. Section~\ref{sec:preliminaries} gives some mathematical preliminaries, followed by Section~\ref{sec:examples} which discusses several examples. Sections~\ref{sec:dfw} and \ref{sec:sfw} introduce and treat the (dFW) and (sFW) recursions for~\eqref{mainopt} and for~\eqref{smainopt}, respectively. Section~\ref{sec:numer} contains a numerical example and Section~\ref{sec:conc} concludes.  




\section{LITERATURE AND PERSPECTIVE.} \label{sec:lit}

The optimization problem in~\eqref{mainopt} is on the space $\mathcal{P}(\mathcal{X})$ of probability measures, that is, $\mathcal{P}(\mathcal{X})$ is the space of (non-negative) measures $\mu$ defined on a measurable space $(\mathcal{X},\Sigma)$ with $\mu(\mathcal{X}) = 1$. The space $\mathcal{P}(\mathcal{X})$ is not a vector space since $\mu_1,\mu_2 \in \mathcal{P}(\mathcal{X})$ does not imply that $\mu_1 + \mu_2 \in \mathcal{P}(X)$, nor that $c \mu \in \mathcal{P}(\mathcal{X})$ for $c \in \mathbb{R}$ and $\mu \in \mathcal{P}(\mathcal{X})$. The natural way to remedy this issue is to extend the objective $J$ in~\eqref{mainopt} to the (Banach) space of \emph{signed measures} $\mathcal{M}(\mathcal{X})$ --- see Definition~\ref{def:measure}. Through such extension, one can in principle leverage the existing standard algorithms for optimizing over the space of signed measures~\cite{2002molzuy,2000molzuy,2022achi,2022bchi,2013brepik}. However, two key complications arise with such an approach. First, extending $J$ to $\mathcal{M}(\mathcal{X})$ meaningfully is often not simple or direct. Second, since the original problem~\eqref{mainopt} is on $\mathcal{P}(\mathcal{X})$, an algorithm that generates iterates on $\mathcal{M}(\mathcal{X})$ might either have to use an implicit projection onto $\mathcal{P}(\mathcal{X})$, or explicitly characterize a descent step that keeps the iterates within $\mathcal{P}(\mathcal{X})$. The former idea of projection is challenging since $\mathcal{P}(\mathcal{X})$ is not a Hilbert space, and there exists no obvious notion of orthogonality. The latter idea has some history --- for instance, Theorem~4.1 in~\cite{2002molzuy} characterizes a step sequence $\eta_{k} \in \mathcal{M}(\mathcal{X}), k \geq 1$ so that the generated iterate sequence $\mu_{k+1} = \mu_k + \eta_k$ continues along the ``steepest descent'' direction while $\mu_{k+1}$ remains in the feasible region, which is $\mathcal{P}(\mathcal{X})$ in the current context. It is important that while the step $\eta_k$ is characterized, actual computation of $\eta_k$ is not easy, making implementation quite intricate. This is why, in the classical optimal design literature~\cite{2012fedhac,1983puk}, the most successfully implemented methods obtain each subsequent iterate $\mu_{k+1}$, not by adding a descent step $\eta_k \in \mathcal{M}(\mathcal{X})$ as in~\cite{2002molzuy}, but by taking the convex combination $\mu_{k+1} = (1-t_k)\mu_k + t_k\nu_k$ where $\mu_k, \nu_k \in \mathcal{P}(\mathcal{X})$. This is a simple strategy to keep the iterates \emph{primal feasible}, that is, within the space $\mathcal{P}(\mathcal{X})$, without having to project or perform intricate step calculations. Indeed the methods we describe next, and those we propose here, use this key idea. (For an example on experimental design, see Example 4.3 in Section~\ref{sec:examples}.) 

Over the previous six years or so, on the heels of the revival elsewhere of a well-known idea called the Frank-Wolfe recursion, a.k.a. the conditional gradient method (CGM)~\cite{frank1956algorithm,1970demrub,1978dunhar,pokutta2023frankwolfe}, there has been interest and success~\cite{eftekhari2019sparse,2017boygeorec} in solving variations of~\eqref{mainopt} where the objective $J$ has a \emph{convex loss of a linear functional} (CLLF) structure: 
\begin{alignat}{2}\label{mainoptspec} \min_{\mu} \quad & J_1(\mu) :=  \ell\bigg(\Phi_{\mu} - y_0\bigg), \quad \Phi_{\mu} = \int_{\Theta} \psi(x) \, \mu(\mathrm{d}x) \nonumber \\ \mbox{subject to} \quad & \|\mu \| \leq \tau; \quad \mu \in \mathcal{M}(\Theta). \tag{$P_0$} 
\end{alignat} In~\eqref{mainoptspec}, $\psi: \Theta \to \mathbb{R}^d$ is differentiable, $\Theta$ is a compact subset of $\mathbb{R}^d$, $\ell: \mathbb{R}^{d} \to \mathbb{R}$ is a convex, differentiable loss function, $\tau>0$ is a positive constant, $\| \mu \|$ refers to the total variation norm of the measure $\mu$, and $y_0 \in \mathbb{R}^{d}$ is a constant vector. The authors in~\cite{2017boygeorec}, apparently the first to  propose CGM to solve such problems, motivate the context with the following ``loss-recovery'' optimization problem in Euclidean space
\begin{alignat}{2}\label{findimstart} \min_{\theta,w,K} \quad & \ell\left(\sum_{j=1}^K w_{j}\Phi(\theta_{j}) - y_0 \right) \nonumber \\ \mbox{subject to} \quad & K \leq N. \tag{$P_1$} \end{alignat} In the above, the decision variables $\theta \in \Theta \subset \mathbb{R}^d$ for $\Theta$ compact and convex, and $N \in \mathbb{N}$ is some known upper bound. In the usual application settings abstracted by~\eqref{findimstart}, $y_0$ is an observed noisy signal, $\theta_i, i=1,2,\ldots,K$ are ``locations'' of sources, $w_i, i=1,2,\ldots,K$ their weights, and $\ell$ a loss function. The formulation~\cite{2017boygeorec} thus looks to identify $\theta,w,K$ that minimizes deviation from the observed signal $y_0$, as quantified through the loss function $\ell$. Noticing that the objective in~\eqref{findimstart} may be nonconvex even if $\ell$ is convex, the authors in~\cite{2017boygeorec} re-frame the problem by \emph{lifting} into the space $\mathcal{M}_K := \{\mu: \mu = \sum_{j=1}^K w_j \delta_{\theta_j}\}$ of weighted atomic measures supported on a finite number of points.

Lifting into the space $\mathcal{M}_K$ is a crucial idea since it endows the CLLF structure to the objective in the lifted space, and allows~\eqref{mainoptspec} to be solved through CGM, whereby each subsequent iterate $\mu_{n+1}$ in the generated solution sequence $\{\mu_n, n \geq 1\}$ is obtained as a convex combination of the incumbent iterate $\mu_n$ and the ``closed form'' solution to a CGM subproblem. In particular,~\cite{2017boygeorec} show that the CGM subproblem amounts to minimizing the linear approximation to $J_1$ at $\mu_n$ over $\theta \in \Theta$, that is, solving
\begin{alignat}{2}\label{cgmsub}
\min_{\theta} \quad & F(\theta) := \left\langle \nabla \ell(\int \Phi(x) \, \mu_n( \mathrm{d}x) - y_0), \Phi(\theta) \right\rangle \nonumber \\ \mbox{subject to} \quad & \theta \in  \Theta. \tag{$P_0$-sub} \end{alignat}
Furthermore,~\cite{2017boygeorec} also argue that that the problem in~\eqref{cgmsub} has the ``closed-form'' solution $-\mbox{sgn}(F(\theta_*)) \, \delta_{\theta^*}$ where $\theta^* = \arg\max | F(\theta) |, \theta \in \Theta$. The closed-form solution $-\mbox{sgn}(F(\theta_*)) \, \delta_{\theta^*}$, apart from allowing the method to approach the solution to the infinite-dimensional problem~\eqref{mainoptspec} through a sequence of finite-dimensional (although nonconvex) subproblems, ensures a simple update scheme whereby a single support point is added to $\mu_n$ to obtain the next iterate $\mu_{n+1}$. Owing to wide applicability, CGM and its variants for solving~\eqref{cgmsub} have become enormously popular over the past six years, since the seminal ideas in~\cite{2017boygeorec}.

\subsection{The Influence Function.}
Is it possible to generalize the ideas of~\cite{2017boygeorec} to operate on $\mathcal{P}(\mathcal{X})$ directly, and to objectives that do not have the CLLF structure? What can be said about the sparsity of solutions obtained through any such generalization? These questions are important because the objectives in problems of the type~\eqref{mainopt}, including the emergency response problem described earlier, do not naturally endow the CLLF structure or sparsity. (See Section~\ref{sec:examples} for additional examples that illustrate this point.) It thus becomes relevant to more closely investigate the extent to which the efficiencies enjoyed by the CGM paradigm are due to the CLLF structure, and whether the requirement for sparsity can be relaxed. 

Indeed, an examination of the calculations leading to the closed-form solution $-\mbox{sgn}(F(\theta_*)) \, \delta_{\theta^*}$ suggests that a certain type of weak differential structure as encoded through the classical \emph{influence function}~\cite{2022shaetal,1983fer,2011fer,2019chublagly} of $J$, as opposed to the CLLF structure, may be the crucial ingredient for efficiency. To explain more precisely, recall that the von Mises derivative $J'_{\mu}(\cdot): \mathcal{P}(\mathcal{X}) \to \mathbb{R}$ associated with a functional $J: \mathcal{P}(\mathcal{X}) \to \mathbb{R}$ is given by
$J'_{\mu}(\nu) = \lim_{\varepsilon \to 0^+} \frac{1}{\varepsilon} \left\{ J( (1-\varepsilon)\mu + \varepsilon \nu ) - J(\mu) \right \}$ if there exists a function $\phi_{\mu}: \mathcal{X} \to \mathbb{R}$ such that $J'_{\mu}(\nu) = \mathbb{E}_{X \sim \nu}[\phi_{\mu}(X)] - \mathbb{E}_{X \sim \mu}[\phi_{\mu}(X)]$. The influence function is defined as $
h_{\mu}(x) = \lim_{\varepsilon \to 0^+} \frac{1}{\varepsilon} \left\{ J( (1-\varepsilon)\mu + {\varepsilon}\delta_x ) - J(\mu) \right \}$, where $\delta_x$ is the Dirac measure (or ``atomic mass'') at $x \in \mathcal{X}$. The von Mises derivative and the influence function of $J$ should be understood as weak forms of a functional derivative for $J$, with the function $\phi_{\mu}$ and the influence function $h_{\mu}$ coinciding to within a constant when the von Mises derivative exists. (See Definition~\ref{vonmises} for more discussion.) 

Now, let's observe that $F(\theta)$ appearing in~\eqref{cgmsub} is indeed the influence function of $J_1$ at $\mu_k$ along $\delta_{\theta} - \mu_k$ since, under sufficient regularity conditions,
\begin{align*}\label{infalready}
\lefteqn{J'_{1,\mu}(\nu) := \lim_{\varepsilon \to 0^+} \frac{1}{\varepsilon} \left\{\ell\left(\int \Phi(x) \, ((1-\varepsilon)\mu + \varepsilon\, \nu)(\mathrm{d}x) - y_0\right) - \ell\left(\int \Phi(x) \, \mu(\mathrm{d}x) - y_0\right) \right\}} \nonumber \\
&= \lim_{\varepsilon \to 0^+} \frac{1}{\varepsilon} \left\{ \ell\left(\int \Phi(x) \, \mu(\mathrm{d}x) - y_0\right) + \left \langle \nabla \ell\left(\int \Phi(x) \, \mu(\mathrm{d}x) - y_0\right) , \int \Phi(x) \varepsilon \, (\nu-\mu)(\mathrm{d}x) \right\rangle  + o(\varepsilon^2) \right.\\
& \qquad \qquad \left. - \ell\left(\int \Phi(x) \, \mu(\mathrm{d}x)-y_0\right)\right\} \nonumber \\ &= \left\langle \nabla \ell\left(\int \Phi(x) \, \mu(\mathrm{d}x) - y_0\right) , \int \Phi(x) (\nu-\mu)(\mathrm{d}x)  \right\rangle.
\end{align*}
Through a similar calculation, we see that the influence function of $J_1$ is
\begin{align} h_{\mu}(x) &= J'_{1,\mu}(\delta_x) \nonumber \\
&=  \left\langle \nabla \ell\left(\int \Phi(z) \, \mu(\mathrm{d}z) - y_0\right) , \Phi(x) - \int \Phi(z) \, \mu(\mathrm{d}z) \right\rangle, \quad x \in \mathbb{R}^d\end{align}
coinciding with $F(\theta)$ in~\eqref{cgmsub} to within an additive constant. 


The influence function in the current context is a weak derivative (or first variation) of a functional defined on a probability space. This suggests that, in the preceding discussion, the existence of a well-behaved influence function, as opposed to the CLLF structure, is the key ingredient when constructing general first-order methods for optimizing over probability spaces. As we shall show through Lemma~\ref{lem:dFW}, the influence function is a succinct and meaningful object with which to express the ``closed-form'' solutions of subproblems appearing within CGM-type analogues for probability spaces. Since the support of an optimal measure to~\eqref{mainopt} is a subset of the set of zeros of the influence function (see Lemma~\ref{lem:supportopt}), the nature of the set of zeros of the influence function or its gradient function often determine whether the optimal measure has sparse support. 

\begin{remark}The influence function is well-recognized as a useful mathematical object in statistics and econometrics, appearing in several incarnations. For instance: (i) as a derivative, it forms the linchpin of the theory of robust statistics especially when measuring the sensitivity of estimators to model misspecification or changes~\cite{1974ham,1996hub,2017andgensha,2009firetal}; (ii) as the summand, when approximating non-linear (but asymptotically linear) estimators using a simple sample mean~\cite{2015sha}, and (iii) for orthogonal moment construction~\cite{2018cheetal,2022cheetal} analogous to the Gram-Schmidt process, especially when debiasing high-dimensional machine-learning estimators. We hasten to add, however, that the influence function may not always exist, and even when it does exist, may only be observed with error (see Sections~\ref{sec:examples} and \ref{sec:sfw}). \demo\end{remark}


\subsection{Why not simply ``grid and optimize''?}\label{sec:grid}

Recall that the problem we consider is an optimization problem over the space $\mathcal{P}(\mathcal{X})$ of probability measures supported on $\mathcal{X}$. The natural way to address the ``infinite-dimensionality'' of $\mathcal{P}(\mathcal{X})$ during computation is to simply ``grid,'' that is, construct a lattice $\mathcal{L}(\Delta)$ having resolution $\Delta>0$ over the compact set $\mathcal{X} \subset \mathcal{R}^d$ and then perform the optimization over the space of probability measures having finite support $\mathcal{L}(\Delta)$. Such a method is sound in that as $\Delta \to 0$, the solution to the (gridded) finite-dimensional approximation can be expected to approach (in some sense) a solution to the optimization problem on $\mathcal{P}(\mathcal{X})$. Furthermore, the obvious advantage of such a strategy is that the power of nonlinear programming on finite-dimensional spaces can immediately be brought to bear.

While the gridding strategy is attractive due to its simplicity, the results are generally poor especially as the resolution size $\Delta$ becomes small, or as the dimension $d$ becomes large. (See, for example, the interesting simple experiment devised in~\cite{eftekhari2019sparse} to illustrate this issue, and also the discussion in~\cite{2017boygeorec}.) The fundamental drawback of gridding is that a \emph{uniform grid} implicitly ignores the structure (e.g., smoothness or convexity) of the objective $J$, and a \emph{non-uniform grid} that adapts to $J$'s structure is very challenging to implement correctly as has been (implicitly) noted in~\cite{eftekhari2019sparse}, and in other infinite-dimensional contexts~\cite{2005canuto,2017ulbzie,2019garulb}.  

The method we propose here circumvents gridding altogether, by constructing a first-order recursion that operates directly in the infinite-dimensional space. The key enabling machinery is the influence function, which when embedded as a first-order variational object within a primal recursion such as Frank-Wolfe, removes the need to \emph{a priori} finite-dimensionalize. The proposed recursion is not without challenge, however, as it entails solving a global optimization problem during each iterate update. The question of precisely characterizing and comparing the complexity of an a priori finite-dimensionalizing approach such as gridding versus the proposed direct approach is indeed interesting, although we do not undertake this question.

\section{PRELIMINARIES.}\label{sec:preliminaries} In this section, we discuss concepts, notation, and important results invoked in the paper. 

\subsection{Definitions.} \begin{definition}[Measure, Signed Measure, Probability Measure]\label{def:measure} Let $(\mathcal{X},\Sigma)$ be a measurable space. A set function $\mu: \Sigma \to \mathbb{R}^+ \cup \{\infty\}$ is called a \emph{measure} if (a) $\mu(A) \geq 0 \,\, \forall A \in \Sigma$, (b)  $\mu(\emptyset)=0$, and (c) $\mu\left(\bigcup_{j=1}^{\infty} A_i \right) = \sum_{j=1}^{\infty} \mu(A_i)$ for a countable collection $\{A_j, j \geq 1\}$ of pairwise disjoint sets in $\Sigma$. The set function $\mu$ is called a \emph{signed measure} if the non-negativity condition in (a) is dropped and the infinite sum in (c) converges absolutely. It is called a \emph{$\sigma$-finite measure} if there exists a countable collection $\{A_j, j \geq 1\}$ such that $\mu(A_j) < \infty, j \geq 1$ and $\bigcup_{j=1}^{\infty} A_j = \mathcal{X}$, and a \emph{probability measure} if $\mu(\mathcal{X}) = 1$. In the current paper $\mathcal{X} \subseteq \mathbb{R}^d$, $\Sigma \equiv \mathcal{B}(\mathcal{X})$ is the Borel $\sigma$-algebra on $\mathcal{X}$, and $\mathcal{P}(\mathcal{X})$ refers to the set of probability measures on $(\mathcal{X},\mathcal{B}(\mathcal{X}))$. \demo

\begin{definition}[Support] The \emph{support} of a probability measure $\mu \in \mathcal{P}(\mathcal{X})$ is the set consisting of points $x$ such that every open neighborhood of $x$ has positive probability under $\mu$. Formally, \begin{equation}\label{suppdef} \mathrm{supp}(\mu) := \bigcup \left\{ x \in \mathcal{X}: \mu(N_x)>0, N_x \mbox{ is any open neighborhood of } x\right\}.
\end{equation} Equivalently, $\mathrm{supp}(\mu)$ is the largest set $C$ such that any open set having a non-empty intersection with $C$ has positive measure assigned to it by $\mu$. The support should be loosely understood as the smallest set such that the measure assigned to the set is one.  \end{definition}

\end{definition}

\begin{definition}[Influence function and von Mises Derivative]\label{vonmises} Suppose $J: \mathcal{P}(\mathcal{X}) \to \mathbb{R}$ is a real-valued function, where $\mathcal{P}(\mathcal{X})$ is a convex space of probability measures on the measurable space $(\mathcal{X},\Sigma)$. There exist various related notions of a derivative of $J$, a few of which we describe next. (See~\cite{1983fer} for more detail.).

The \emph{influence function} $h_{\mu}: \mathcal{X} \to \mathbb{R}$ of $J$ at $\mu \in \mathcal{P}(\mathcal{X})$ is defined as \begin{equation}\label{influence} h_{\mu}(x) = \lim_{t \to 0^+}\frac{1}{t} \bigg\{ J(\mu + t(\delta_x - \mu)) - J(\mu) \bigg\}, \end{equation} where $\delta_x := \mathbb{I}_{A}(x), A \subset \mathcal{X}$ is the Dirac measure (or atomic mass) concentrated at $x \in \mathcal{X}$~\cite{1983fer,2011fer}. The influence function should be loosely understood as the rate of change in the objective $J$ at $\mu$, due to a perturbation of $\mu$ by a Dirac measure (point mass) $\delta_x$.

The \emph{von Mises derivative} is defined as
$$J'_{\mu}(\nu) := \lim_{t \to 0^+}\, \frac{1}{t} \bigg\{J(\mu + t(\nu - \mu)) - J(\mu)\bigg\}, \quad \mu,\,\nu\in \mathcal{P}(\mathcal{X}),
$$
provided $J'_{\mu}(\cdot)$ is \emph{linear} in its argument, that is, there exists a function $\phi_{\mu}: \mathcal{X} \to \mathbb{R}$ such that \begin{align} \label{vonlinear} J'_{\mu}(\nu) &= \int \phi_{\mu}(x)\, \mathrm{d}(\nu - \mu)(x) \\ &=: \mathbb{E}_{X \sim \nu}[\phi_{\mu}(X)] - \mathbb{E}_{X \sim \mu}[\phi_{\mu}(X)]. \nonumber \end{align} When~\eqref{vonlinear} holds, we can see that $\phi_{\mu}$ in~\eqref{vonlinear} and $h_{\mu}$ in~\eqref{influence} coincide to within a constant since $\mathrm{d}(\nu - \mu)$ has total measure zero. As implied by Lemma~\ref{lem:infexists} in Section~\ref{sec:supp}, the influence function is a weak form of a derivative. It is strictly weaker than the von Mises derivative in the sense that it exists if the von Mises derivative exists, but the converse is not true --- see Example 2.2.2 in~\cite{1983fer}. \demo \end{definition}

\begin{definition}[G\^{a}teaux, Fr\'{e}chet and Hadamard Derivatives]\label{frechet} 
To understand the influence function's relationship to other (stronger forms of) functional derivatives, suppose $J$ is defined on an open subset of a normed space that contains $\mu$. A \emph{continuous linear} functional $J'(\cdot;\mu)$ is a \emph{derivative} of $J$ if \begin{equation} \label{funcder} \lim_{t \to 0^+} \frac{1}{t} \bigg\{J(\mu + t(\nu - \mu)) - J(\mu) \bigg\} - J'(\nu - \mu; \mu)  = 0\end{equation} for $\mu$ in subsets of the domain of $J$. Various functional derivatives are defined depending on how~\eqref{funcder} holds. For instance, if~\eqref{funcder} holds pointwise in $\mu$, then $J'(\cdot;\mu)$ is called a \emph{G\^{a}teaux derivative}; it is called a \emph{Hadamard derivative} if~\eqref{funcder} holds uniformly on compact subsets of the domain of $J$, and a \emph{Fr\'{e}chet derivative} if~\eqref{funcder} holds uniformly on bounded subsets of the domain of $J$. Accordingly, Fr\'{e}chet differentiability is the most stringent and implies Hadamard differentiability, which in turn implies Gateaux differentiability. In each case, the influence function is the central ingredient since, from~\eqref{vonlinear}, we have \begin{align}\label{central}J'_{\mu}(\nu) &= \int \phi_{\mu}(x) \, \mathrm{d}(\nu - \mu)(x) \nonumber \\ &=  \mathbb{E}_{X\sim \nu}[h_{\mu}(X)],\end{align} {where the second equality follows from the fact that the influence function can be expressed as  \[
h_{\mu}(x) = \int \phi_{\mu}(y) \, \mathrm{d}(\delta_x - \mu)(y) = \phi_{\mu}(x) - \int \phi_{\mu}(y) \, \mathrm{d}\mu(y),
\]which implies  \[
\int h_{\mu}(x) \, \mathrm{d}\mu(x) =\int \left( \phi_{\mu}(x) - \int \phi_{\mu}(y) \, \mathrm{d}\mu(y) \right) \mathrm{d}\mu(x) = 0.
\]} \demo \end{definition}

A complicating aspect of the problem considered in this paper is that the stronger forms of the derivative as defined through~\eqref{funcder} require a vector space as the domain of $J$ (since $t$ can approach zero from either side) whereas the space $\mathcal{P}(\mathcal{X})$ of probability measures is not a vector space. The space $\mathcal{P}(\mathcal{X})$ can be extended to form a vector space through the definition of signed measures but, as we shall see, the structure of the Frank-Wolfe recursion that we consider is such that it obviates such a need, while also allowing the use of a weaker functional derivative such as the influence function. 

\begin{definition}[{\em L}-Smooth]\label{def:Lsmootn} The functional $J: \mathcal{P}(\mathcal{X}) \to \mathbb{R}$ is $L$-smooth with constant $L$ if it satisfies \begin{align}\label{Lsmooth}
    \sup_{x \in \mathcal{X}} |h_{\mu_1}(x)-h_{\mu_2}(x) | \, & \leq \, L \|\mu_1 - \mu_2 \|, \quad \forall \mu_1 ,\, \mu_2 \in \mathcal{P}(\mathcal{X}),
\end{align} where $h_{\mu_1}$ and $h_{\mu_2}$ are corresponding influence functions, and the total variation distance between $\mu_1$ and $\mu_2$ in $\mathcal{P}(\mathcal{X})$ is defined as
\begin{align}\label{totvar}
    \|\mu_1 - \mu_2 \| &:=\sup_{A\in \mathcal{B}(\mathcal{X})} |\mu_1(A)-\mu_2(A) |,
\end{align} where $\mathcal{B}(\mathcal{X})$ is the Borel $\sigma$-algebra.  As written, the symbol $\| \cdot \|$ appearing in~\eqref{totvar} does not refer to a norm but our use of such notation is for convenience and should cause no confusion.\demo
\end{definition} 

\subsection{Basic Properties}\label{sec:supp} We list some basic properties that are directly relevant to the content of the paper. The ensuing result asserts that if $J$ in~\eqref{mainopt} is convex, then a point $\mu^* \in \mathcal{P}(\mathcal{X})$ is optimal if and only the influence function at $\mu^*$ is non-negative. This result justifies the analogous roles that the influence function and the directional derivative play in the respective contexts of optimization over probability space and the Euclidean space.

\begin{lemma}[Conditions for Optimality]\label{lem:optCondition} Suppose $J: \mathcal{P}(\mathcal{X}) \to \mathbb{R}$ is convex, and the von Mises derivative exists at $\mu^* \in \mathcal{P}(\mathcal{X})$ along any ``direction'' $\nu - \mu^*$ where $\nu \in \mathcal{P}(\mathcal{X})$. The influence function $h_{\mu^*}$ at $\mu^*$ is defined in~\eqref{influence}. Then, $\mu^*$ is optimal, that is, $J(\nu) \geq J(\mu^*)$ $\forall \nu \in \mathcal{P}(\mathcal{X})$ if and only if $h_{\mu^*}(x) \geq 0$ for all $x \in \mathcal{X}$.
\end{lemma}
\begin{proof}{Proof.} Suppose that $h_{\mu^*}$ is non-negative. Since $J$ is convex, we see that \begin{align*} J(\nu) &\geq J(\mu^*) + J'_{\mu^*}(\nu) \\ &= J(\mu^*) + \int_{\mathcal{X}} h_{\mu^*}(x) \, \nu( \mathrm{d}x) \geq J(\mu^*),\end{align*} where the last inequality holds because $h_{\mu^*}(x) \geq 0$ for all $x$ and $\nu \in \mathcal{P}(\mathcal{X})$. Hence, $\mu^*$ is optimal. 

Now, let's prove the converse statement. Let $\mu^*$ be optimal. If there exists $x_0 \in \mathcal{X}$ such that $h_{\mu^*}(x_0)<0$, then \begin{equation*}
    0>h_{\mu^*}(x_0) =\lim_{t \to 0^+}\frac{1}{t} \bigg\{ J(\mu^* + t(\delta_{x_0} - \mu^*)) - J(\mu^*) \bigg\} \geq 0,
\end{equation*} where the last inequality follows from $J(\mu^* + t(\delta_{x_0} - \mu^*)) \geq J(\mu^*)$ for all $t \in [0,1]$. Thus, we obtain a contradiction. \demo
\end{proof}

The next lemma is intended to shed light on the sparsity of a solution to~\eqref{mainopt}. In particular, the lemma exposes a connection between the nature of the set of zeros of the influence function at an optimal point, and the support of the optimal point. See especially Example 4.2 in Section~\ref{sec:examples} for more insight on how sparsity emerges.

\begin{lemma}[Support of Optimal Measure]\label{lem:supportopt} Suppose $\mu^*$ is optimal and the von Mises derivative exists at $\mu^*$. The support of $\mu^*$ satisfies \begin{equation}\label{suppinfl}
    \mathrm{supp}(\mu^*) \subseteq \left\{ x \in \mathcal{X}\,:\,h_{\mu^*}(x)=0 \right\} \quad \mu^* \mbox{ a.s.\ }
\end{equation} Moreover, if $h_{\mu^*}$ is differentiable, we have  \begin{equation}\label{suppdinfl}
    \mathrm{supp}(\mu^*) \subseteq \left\{ x \in \mathcal{X}\,:\,\nabla h_{\mu^*}(x)=0 \right\} \quad \mu^* \mbox{ a.s.\ }
\end{equation}
\end{lemma}
\begin{proof}{Proof.} Since $\mu^*$ is optimal, as per Lemma~\ref{lem:optCondition} \begin{equation*}
    h_{\mu^*}(x) \geq 0 \qquad \forall x \in \mathcal{X}.
\end{equation*} Assume there exists a nonempty set $A\subseteq \mathrm{supp}(\mu^*)$ such that $h_{\mu^*}(x)>0$ for all $x \in A$. From~\eqref{central}, we know $\mathbb{E}_{X \sim \mu^*}[h_{\mu^*}(X)]=0$. Then we have the contradiction
\begin{equation*}
    0=\int_{\mathcal{X}} h_{\mu^*}(x) \, \mu^*( \mathrm{d}x) \geq \int_{A} h_{\mu^*}(x) \, \mu^*( \mathrm{d}x) > 0.
\end{equation*} Therefore, $\mathrm{supp}(\mu^*) \subseteq \left\{ x \in \mathcal{X}\,:\,h_{\mu^*}(x)=0 \right\}$. Any $x \in \mathcal{X}$ satisfying $h_{\mu^*}(x)=0$ is a minimum of $h_{\mu^*}$. Hence, if $h_{\mu^*}$ is differentiable, $h_{\mu^*}(x)=0$ implies $\nabla h_{\mu^*}(x)=0$, validating the assertion in~\eqref{suppdinfl}. \demo
\end{proof} 

Lemma~\ref{lem:supportopt} implies that if the set $\left\{ x \in \mathcal{X}\,:\,h_{\mu^*}(x)=0 \right\}$ or $\left\{ x \in \mathcal{X}\,:\,\nabla h_{\mu^*}(x)=0 \right\}$ is ``sparse," meaning that it is a countable set, then the optimal solution $\mu^*$ is also sparse. Conversely, if an optimal measure $\mu^*$ to problem~\eqref{mainopt} is supported on a set $A$, then $A$ is a subset of $\left\{ x \in \mathcal{X}\,:\,h_{\mu^*}(x)=0 \right\}$ or $\left\{ x \in \mathcal{X}\,:\,\nabla h_{\mu^*}(x)=0 \right\}$. These observations are not conclusive about the nature of the support of $\mu^*$ but they nevertheless offer insight.

The next result provides sufficient conditions under which the influence function exists. In particular, it asserts that if $J$, extended to the vector space $\mathcal{M}(\mathcal{X})$ of signed measures, is convex, then the influence function necessarily exists. The result is stated with $J$ extended to $\mathcal{M}(\mathcal{X})$ since the space $\mathcal{P}(\mathcal{X})$ has no interior. We include a proof, but it follows by retracing the classic proof of showing that a convex function (with domain in $\mathbb{R}^d$) has a directional derivative~\cite{2004nes}.

\begin{lemma}[Influence Function Existence]\label{lem:infexists} Suppose $J: \mathcal{M}(\mathcal{X}) \to \mathbb{R}$ is convex. Then, the influence function $h_{\mu}$ given by~\eqref{influence} exists and is finite at each $\mu \in \mathcal{M}(\mathcal{X})$.
\end{lemma} 
\begin{proof}{Proof.} {Consider the function
\begin{equation}
    s(t) := \frac{1}{t} \bigg\{J((1-t)\mu + t \delta_x) - J(\mu) \bigg\}, \quad 0 < t \leq 1.
\end{equation}}Since $J$ is convex, we see that for $0 < \beta, t \le 1$, $$J\left((1-\beta)\mu + \beta((1-t)\mu + t\delta_x)\right) \leq (1-\beta)J(\mu) + \beta J((1-t)\mu + t\delta_x),$$
i.e., that
$$J\left((1-\beta t)\mu + \beta t\delta_x\right) \leq (1-\beta)J(\mu) + \beta J((1-t)\mu + t\delta_x),$$
and so, rearranging,
\begin{align} \label{dirderpr}
s(\beta t) = \frac{1}{\beta t} \bigg\{J((1-\beta t)\mu + \beta t \delta_x) - J(\mu) \bigg\} \leq \frac{1}{t}\bigg\{ J((1-t)\mu + t\delta_x) - J(\mu) \bigg\}.
\end{align}
We see from~\eqref{dirderpr} that $s$ is non-decreasing to the right of zero. 

{ Next, consider \(t_0 > 0\). Since \(\mathcal{M}(\mathcal{X})\) is a vector space, it follows that \(\mu - t_0 (\delta_x - \mu) \in \mathcal{M}(\mathcal{X})\). Furthermore, for any \(t_0 > 0\), we can express \(\mu\) as a convex combination:
\[
\mu = \frac{t_0}{t + t_0} (\mu + t (\delta_x - \mu)) + \frac{t}{t + t_0} (\mu - t_0 (\delta_x - \mu)).
\]
Using the convexity of \(J\), this yields
\[
J(\mu) \leq \frac{t_0}{t + t_0} J(\mu + t (\delta_x - \mu)) + \frac{t}{t + t_0} J(\mu - t_0 (\delta_x - \mu)).
\]
Dividing through by \(\frac{t t_0}{t + t_0}\), we obtain
\begin{equation}
    \left(\frac{1}{t} + \frac{1}{t_0}\right) J(\mu) \leq \frac{1}{t} J(\mu + t (\delta_x - \mu)) + \frac{1}{t_0} J(\mu - t_0 (\delta_x - \mu)).
\end{equation}
This inequality provides a lower bound for \(s(t)\):
\begin{align} \label{dirderpr2}
    s(t) \geq \frac{1}{t_0} \bigg\{J(\mu) - J(\mu - t_0 (\delta_x - \mu)) \bigg\}.
\end{align} Conclude from~\eqref{dirderpr} and~\eqref{dirderpr2} that $\lim_{t \to 0^+} s(t)$ exists and hence the assertion holds.}
\demo    
\end{proof}

The next lemma provides some sufficient conditions under which the influence function $h_{\mu}$ is continuous on the set $\mathcal{X}$, with the implication that $h_{\mu}$ attains its minimum on $\mathcal{X}$.  

\begin{lemma}[Influence Function Continuity]\label{lem:infcont}
Suppose the influence function $h_{\mu}$ of $J: \mathcal{P}(\mathcal{X}) \to \mathbb{R}$ exists at each $\mu \in \mathcal{P}(\mathcal{X})$. Suppose also that $J$ satisfies the following ``steepness'' restriction: for each fixed $\mu \in \mathcal{P}(\mathcal{X})$, each fixed $x \in \mathcal{X}$, and small enough $t>0$, \begin{equation}\label{JLiplike} \bigg|J((1-t)\mu + t\delta_x) - J((1-t)\mu + t\delta_y) \bigg| \leq t\, o(\|x - y\|), \quad y \in \mathcal{X}.\end{equation} Then the influence function $h_{\mu}$ of $J$ is continuous on $\mathcal{X}$.
\end{lemma} \begin{proof}{Proof.} Consider any point $x \in \mathcal{X}$, and let $\{x_n, n \geq 1\}$ be any sequence in $\mathcal{X}$ such that $x_n \to x$. Notice that \begin{align}\label{hcont} \left|h(x_n) - h(x) \right| &= \left| \left\{\lim_{t \to 0^+} \frac{1}{t}\left(J((1-t)\mu + t\delta_{x_n}) - J(\mu)\right)\right\} - \left\{\lim_{t \to 0^+} \frac{1}{t}\left(J((1-t)\mu + t\delta_{x}) - J(\mu)\right)\right\} \right| \nonumber \\
&= \left|\lim_{t \to 0^+} \frac{1}{t}\left(J((1-t)\mu + t\delta_{x_n}) - J((1-t)\mu + t\delta_{x})\right)\right| \nonumber \\ &\leq o(\|x_n - x\|),
\end{align} where the inequality follows from~\eqref{JLiplike}. Since the sequence $\{x_n, n \geq 1\}$ is arbitrary, conclude from~\eqref{hcont} that $h_{\mu}$ is continuous at $x$. \demo
\end{proof}

As we see later (in Section~\ref{sec:dfw}), one of the key aspects of this paper is a recursive algorithm that updates the incumbent solution during each iteration using a Dirac measure located at a minimum of the influence function. Lemma~\ref{lem:infcont} is intended to provide some insight (through the application of the Bolzano-Weierstrass theorem~\cite{1976bar}) on the conditions under which such a minimum is guaranteed to exist.


\section{EXAMPLES.}\label{sec:examples} 

To further intuition, we now provide a number of examples subsumed by~\eqref{mainopt} or~\eqref{smainopt}. These examples illustrate problems that (i) may or may not have the CLLF structure; (ii) have solutions that may be sparse or non-sparse; (iii) have influence functions that are calculable, but not necessarily observable through a deterministic oracle; and (iv) have influence functions observable through an unbiased stochastic oracle. 

{Example~\ref{sec:pmeans} (the P-means problem) provides a meaningful case where the objective is convex but does not exhibit the CLLF structure, while the influence function remains accessible. This highlights the broader applicability of our approach beyond CLLF settings. Additionally, we have verified that under certain conditions, Examples~\ref{sec:Calibration}, \ref{sec:optRes}, \ref{sec:pmeans}, \ref{sec:networks}, and \ref{sec:deconvolution} satisfy the \(L\)-smoothness assumption introduced in Definition~\ref{def:Lsmootn}. For clarity, we explicitly state the additional regularity conditions required for \(L\)-smoothness at the end of these examples.}

{We begin by introducing two stylized examples (Sections~\ref{sec:Calibration} and \ref{sec:optRes}) to illustrate the concept of influence functions. While these problems can be solved directly, they serve as a foundation for understanding the methodology before moving to more complex settings.}

\subsection{Calibration.}\label{sec:Calibration} Consider the question of identifying a probability measure $\mu$ that makes the expected cost of a random variable (distributed as $\mu$) as close to a specified target $y_0 \in \mathbb{R}$ as possible. Formally,
\begin{align}\label{matching} \mbox{min.} & \quad J(\mu) := \left(\int_{\mathcal{X}} f(x)\, \mu(\mathrm{d}x) - y_0 \right)^2  \nonumber \\ \mbox{s.t.} & \quad \mu \in \mathcal{P}(\mathcal{X}),
\end{align} where $y_0 \in \mathbb{R}$ is a real-valued constant, $\mathcal{X} : [a,b]$ is a compact interval of $\mathbb{R}$, and $f: \mathbb{R} \to \mathbb{R}$ is continuous on $\mathcal{X}$. ($f$ attains its maximum and its minimum on $\mathcal{X}$ since $f$ is continuous and $\mathcal{X}$ is compact.) Assume also that $\min_{x \in \mathcal{X}} f(x) \leq y_0 \leq \max_{x \in \mathcal{X}} f(x)$. We can show after some algebra that 
\begin{align*} J'_{\mu}(\nu) &= 2  \left( \int f \mathrm{d}(u- \mu)\right)\left( \int f \mathrm{d}\mu - y_0\right), 
\end{align*}
and the influence function
\begin{align}\label{calibinf}
h_{\mu}(x) &= 2  \left( f(x) - \int f \mathrm{d}\mu \right)\left( \int f \mathrm{d}\mu - y_0\right) \nonumber \\
&= 2\left( \int f \mathrm{d}\mu - y_0\right)f(x) - 2\int f \mathrm{d}\mu \left( \int f \mathrm{d}\mu - y_0\right).
\end{align} Notice that the influence function $h_{\mu}$ in~\eqref{calibinf} has the simple form $$h_{\mu}(x) = c_1(\mu)f(x) + c_2(\mu),$$ with the implication that
\begin{equation} \argmin_{x \in \mathcal{X}}\, h_{\mu}(x) = \begin{cases} \argmin_{x \in \mathcal{X}} f(x), & \int f \mathrm{d}\mu - y_0 > 0; \\
\argmax_{x \in \mathcal{X}} f(x), & \int f \mathrm{d}\mu - y_0 < 0; \\ 
\mathcal{X}, & \mbox{otherwise.}\end{cases}
\end{equation} Conclude that $J^* = \min_{\mu \in \mathcal{P}(\mathcal{X})} J(\mu) =0$ and this minimum value is attained at $$\mu^* = p \delta_{\argmax_{x \in \mathcal{X}} f(x)} + (1-p) \delta_{\argmin_{x \in \mathcal{X}} f(x)},$$ where $p$ is such that $p\max_{x \in \mathcal{X}} f(x) + (1-p) \min_{x \in \mathcal{X}} f(x) = y_0.$ This optimal solution is sparse since it is a mixture of two point probability masses. {Furthermore, we can confirm that the \( L \)-smoothness assumption in Definition~\ref{def:Lsmootn} holds in this example when \( f \) is bounded on \( \mathcal{X} \). Given the quadratic structure of \( J \) and the corresponding influence function, the \( L \)-smoothness condition can be directly verified using Definition~\ref{def:Lsmootn}. For brevity, we omit the proof of this verification.} \hfill \demo

\subsection{Optimal Response Time.}\label{sec:optRes} Our next example consists of two parts (a) and (b), {the first of which illustrates a seemingly common setting where the influence function $h_{\mu}(x)$ of the objective function $J$ in problem~\eqref{mainopt} is constant with respect to the decision variable $\mu$ in the term where $x$ is appearing.}
In such a case, a solution $\mu^*$ to~\eqref{mainopt} simply puts all its mass in the set $\mathcal{X}_{\mu}^* = \mathcal{X}^*:= \arg\min_{x \in \mathcal{X}} h_{\mu}(x)$, assuming that this set is non-empty. This immediately leads to a sparse solution since $\mu^*$ can be set to a point mass on any element of $\mathcal{X}^*$. In part (b), a slight variation illustrates a setting where the influence function $h_{\mu}$ is not constant in $\mu$, implying that $\mathcal{X}_{\mu}^*$ retains its dependence on $\mu$. More importantly, it easily yields a solution $\mu^*$ that is non-sparse.

\textbf{Part (a):} Consider a ``one-dimensional compact city'' $\mathcal{X}:=[a,b]$ where incidents occur according to a probability measure $\eta \in \mathcal{P}(\mathcal{X}).$ We would like to locate an emergency response vehicle on $\mathcal{X}$ according to the measure $\mu$ in such a way that the ``cost'' (defined appropriately) due to attending to the next incident is minimized. In a simple formulation of this problem, we wish to solve:
\begin{align}\label{optresp} \mbox{min.} & \quad J(\mu) := \int_{\mathcal{X}} c_{\mu}(y) \, \mathrm{d}\eta(y) \nonumber \\ \mbox{s.t.} & \quad \mu \in \mathcal{P}(\mathcal{X}), \end{align}
where \begin{equation} \label{costex} c_{\mu}(y) := \int_{\mathcal{X}} t(x,y) \, \rmd\mu(x), \quad y \in \mathcal{X} \end{equation} and $t: \mathcal{X} \times \mathcal{X} \to [0,\infty)$ represents the cost of a response from $x$ to $y$, $t(\cdot,y)$ continuous for each $y \in \mathcal{X}$. Under the cost structure in~\eqref{costex}, we get \begin{align}\label{optcalc}
\lim_{\epsilon\to0^+}\frac{1}{\epsilon} \left\{J( (1-\epsilon)\mu + \epsilon \nu) - J(\mu) \right\} &= \int \lim_{\epsilon\to0^+}\frac{1}{\epsilon} \left\{\int t(x,y) \, \mathop{\rmd((1-\epsilon)\mu + \epsilon \nu)(x)} - \int t(x,y) \, \mathop{\rmd \mu(x)}\right\} \, \mathop{\rmd\eta}(y) \nonumber \\
&= \int \int t(x,y) \, \mathop{\rmd}(\nu-\mu)(x) \, \rmd \eta(y) \nonumber \\
&= \mathbb{E}_Y\bigg[\mathbb{E}_{X \sim \nu}[t(X,Y)] -\mathbb{E}_{X \sim \mu}[t(X,Y)]\bigg].
\end{align}
Using the expression in~\eqref{optcalc}, we see that the influence function is \begin{equation}\label{infresp}h_{\mu}(x) = \mathbb{E}_Y[t(x,Y)] - \mathbb{E}_{Y}\left[\mathbb{E}_{X \sim \mu}[t(X,Y)]\right],\end{equation}
Since the dependence of $\mu$ appears as an additive function, we see that $\mathcal{X}^*_{\mu}$ does not depend on $\mu$. Since $t(\cdot,Y)$ is continuous $Y$-almost surely and $\mathcal{X}$ is compact, the set $\mathcal{X}^* := \arg\min_{x \in \mathcal{X}} h_{\mu}(x)$ is non-empty. Let $\mu^*$ be a measure supported on $\mathcal{X}^*$. Then, since $J$ in~\eqref{optresp} is linear (and hence convex) in $\mu$, we see that for any $\nu \in \mathcal{P}(\mathcal{X})$,
\begin{align*}
J(\nu) &= J(\mu^*) + \int h_{\mu^*}(x) \, \mathop{\rmd}(\nu - \mu^*)(x) \\
& \geq 
J(\mu^*) + \int_{\mathcal{X}} \min_{y \in \mathcal{X}} h_{\mu^*}(y) \, \mathop{\rmd}\nu(x) - \int_{\mathcal{X}^*} h_{\mu^*}(x) \mathop{\rmd}\mu^*(x) \\
& = J(\mu^*),\end{align*} implying that $\mu^*$ is optimal. {Moreover, due to the linear structure of the influence function \( h_{\mu} \), a sufficient condition for \( J \) to satisfy the \( L \)-smoothness assumption is that \( t \) remains bounded.}

\textbf{Part (b)}: Consider now the following simple variation. Suppose $F_{\mu}(t), t \in [0,\infty)$ represents the ``probability of the response time to a random incident being at most $t$,'' assuming the response vehicle location $x \sim \mu$ and incident  location $Y \sim \eta$ are independent, and that the response vehicle moves at constant speed $v$:
\begin{equation}
F_{\mu}(t) = \int \int I(|x-y| \leq vt)\, \mathop{\rmd}\mu(x) \mathop{\rmd}\eta(y).
\end{equation}
Suppose now that we seek a measure $\mu$ that makes the resulting $F_{\mu}$ ``closest'' to a target profile curve $F^*(t), t \in [0,\infty)$ in squared error, that is, we seek to solve:
\begin{align}\label{optrespb}
\mbox{min.} & \quad J(\mu) := \int_0^{\infty} (F_{\mu}(t) - F^*(t))^2\, \mathop{\rmd}t \nonumber \\
\mbox{s.t.} & \quad \mu \in \mathcal{P}(\mathcal{X}).
\end{align}
Algebra similar to that used in the previous part then yields that
\begin{equation}\label{infrespb} h_{\mu}(x) = 2 \int_0^{\infty} \left( F_{\mu}(t) - F^*(t)\right) \, [P\left(\left| Y - x \right| \leq vt \right) - P\left(\left| Y - X \right| \leq vt \right)]\, \mathop{\rmd}t,
\end{equation}
where the incident location $Y \sim \eta$ and the volunteer location $X \sim \mu$. Unlike in part (a), the set $\mathcal{X}_{\mu}^* := \arg\min h_{\mu}(x)$ is intimately dependent on $\mu$, and perhaps more importantly, even simple choices of $\eta$ and $F^*$ can yield non-sparse solutions $\mu^*$. For example, suppose $v=1$, $\mathcal{X} = [0,1]$ and $$\eta = \delta_{\frac{1}{2}}; \quad  F^*(t) = \begin{cases} 2t & 0 \leq t \leq \frac{1}{2} \\ 1 & t > \frac{1}{2}. \end{cases} $$ Then simple calculations yield that $F_{\mu^*}(t) = F^*(t)$ for all $t \in [0,\infty)$ and $h_{\mu^*}(x) = 0$ if $\mu^* = \mbox{Unif}(0,1).$ In fact, we can show that any finitely supported sparse solution has to be sub-optimal. Consider any finitely supported sparse solution $\tilde{\mu} = \sum_{i=1}^n p_i \delta_{x_i}$, where $p_i>0, \sum_{i=1}^n p_i =1$, $x_i \in [0,1], i=1,2,\ldots, n$ are distinct, and $0 \leq x_1 < x_2 < \cdots < x_n \le 1$ without loss of generality. If $x_j \ne 1/2$ for all $j$ then $F_{\tilde \mu}(t) = 0$ for $0 \le t < \min\{|x_j-1/2|, j = 1, 2, \ldots, n\}$. Otherwise, if $x_{j^*} = 1/2$, then $F_{\tilde \mu}(t) = p_{j^*}$ for $0 \le t < \min\{|x_j - 1/2|, j=1,2,\ldots,n, j \ne j^*\}$.
We then see that $J(\tilde{\mu}) = \int_0^{\infty} (F_{\tilde{\mu}}(t) - F^*(t))^2 \, \mathop{\rmd}t >0$ implying that $\tilde{\mu}$ is sub-optimal. { In addition, we can verify that a sufficient condition for \( J \) to be \( L \)-smooth is the boundedness of \( F^* \).}\hfill \demo

\subsection{Optimal Experimental Design} Consider the question how best to sample points from a space $\mathcal{X} \subseteq \mathbb{R}^d$ when estimating the parameter vector $\beta^* \in \mathbb{R}^d$ of a \emph{regression model} having the form
\begin{equation}\label{regression} Y(X) = f(X)^T\beta^* + \epsilon(X).
\end{equation}
In~\eqref{regression}, $Y(X)$ is called the \emph{response} at $X$, $f = (f_1,f_2,\ldots,f_d): \mathcal{X} \to \mathbb{R}^d$  is a vector of \emph{orthogonal} real-valued functions on $\mathcal{X}$, that is, $\int_{\mathcal{X}} f_i(x)f_j(x) \, \mathrm{d}x = 0$ for $i \neq j$, and $\epsilon(X)$ satisfies $\mathbb{E}[\epsilon(X)] = 0, \mbox{var}[\epsilon(X)] = \sigma^2$. Now, suppose we observe the responses $Y_1, Y_2, \ldots, Y_n$ at the observation points $X_1,X_2, \ldots,X_n \overset{\text{\tiny{iid}}}{\sim} \mu \in \mathcal{P}(\mathcal{X})$. Let $\hat{\beta}_n$ be the least-squares estimator of $\beta^*$, that is, \begin{equation}\label{ls} \hat{\beta}_n := \arg\min \left\{\sum_{j=1}^n \left(Y_i - f(X_i)^T\beta \right)^2: \beta \in \mathbb{R}^d\right\}.\end{equation} It is known~\cite{2002molzuy} that the covariance $\mbox{cov}(\hat{\beta}_n) = \sigma^2M^{-1}(\mu)$ where
\begin{equation}\label{covmat} M(\mu) := \int_{\mathcal{X}} f(x)f(x)^T \, \mu(\mathrm{d}x)\end{equation} is called the \emph{information matrix}. Various classical experimental designs, e.g., A-optimal, E-optimal, L-optimal, D-optimal~\cite{1974kie}, seek to maximize or minimize some function of $M^{-1}(\mu)$ with respect to $\mu$ in an attempt to identify a good design. For instance, the most widely used $D$-optimal design seeks to solve:
\begin{align}\label{expdes} \mbox{min.} & \quad J(\mu) :=  \det M^{-1}(\mu)\, \nonumber \\ \mbox{s.t.} & \quad \mu \in \mathcal{P}(\mathcal{X}). \end{align} The problem in~\eqref{expdes} is indeed an optimization problem over the space of probability measures. {Since the differential of the determinant satisfies $\mathrm{d} \det(M) = \det(M) \mathrm{tr} \left\{ M^{-1} \mathrm{d}M \right\}$ by Theorem 8.1 in~\cite{magnus2019matrix}, and the von Mises derivative of \( M(\mu) \) along \( \nu - \mu \) is given by \( M(\nu) - M(\mu) \), applying the chain rule yields the von Mises derivative of \( J \):} $$J'_{\mu}(\nu) = -\left( \det M^{-1}(\mu) \right)\mathrm{tr}\bigg\{M^{-1}(\mu) (M(\nu) - M(\mu)) \bigg\},$$ implying the influence function \begin{align} h_{\mu}(x) &= -\mathrm{tr}\bigg\{ M^{-1}(\mu) f(x)f(x)^T -I_d\bigg\} \, \det M^{-1}(\mu) \nonumber \\ &= \bigg( -f(x)^T M^{-1}(\mu) f(x) + d \bigg) \, \det M^{-1}(\mu), \quad x \in \mathcal{X}.
\end{align} Through a similar procedure, influence function expressions for other classical experimental designs can be obtained. {Since \( f(y) f(y)^T \) for \( y \in \mathcal{X} \) is a rank-one matrix and therefore not invertible, the influence function becomes unbounded at \( \mu = \delta_y \), implying that the \( L \)-smoothness assumption fails in this case.
}
\demo

\subsection{{\em P}-means Problem}\label{sec:pmeans}
The $P$-means problem~\cite{2002molzuy,2000oka}
is sometimes called the \emph{randomized variant} of the $k$-means clustering problem. Suppose \emph{demand sources} located at $\ell_1,\ell_2, \ldots,\ell_{n_0} \in \mathcal{X} \subset \mathbb{R}^d$ are to be serviced by \emph{responders} located in $\mathcal{X}$, where $\mathcal{X}$ is a compact set. As part of the randomization, suppose that the responders are located in $\mathcal{X}$ according to a spatial Poisson process $X$ having mean measure $\mu$. Assume for simplicity that $\mu(\mathcal{X}) = 1$, so that $\mu \in \mathcal{P}(\mathcal{X})$. Also, assume that each demand source is serviced by the responder closest to it, that is, for a realization $(X_1, X_2, \ldots, X_N)$ of $X$, the cost incurred due to serving the $i$-th demand, $i=1,2,\ldots,n_0$, is
\begin{equation} c_i(X) = \begin{cases} \min_j \{\|\ell_i - X_j\|, j = 1,2,\ldots,N\} & N \geq 1; \\
u & \mbox{ otherwise},\end{cases}
\end{equation}
where $u = \sup\{\|x_1-x_2\|, x_1,x_2 \in \mathcal{X}\}$ is a fixed constant. (Due to the choice of the constant $u$, $c_i(X) \leq u$ if $N \geq 1$, and $c_i(X) = u$ if $N = 0$.) The $P$-means problem then seeks a $\mu \in \mathcal{P}(\mathcal{X})$ that minimizes the expected total cost \begin{align}\label{totcost_pmeans} J(\mu) &= \sum_{i=1}^{n_0} \, \int_0^{\infty} P_X\left( c_i(X) > t \right) \, \mathrm{d}t \nonumber \\
&= \sum_{i=1}^{n_0} \, \int_0^{u}  \exp\left\{-\mu(B(\ell_i,t))\right\} \, \mathrm{d}t, \quad \mu \in \mathcal{P}(\mathcal{X}).
\end{align}
Some algebra starting from first principles then gives the influence function of $J$:\begin{equation}\label{inf_pmeans} h_{\mu}(x) = -\sum_{i=1}^{n_0} \, \int_0^{u} \mathbb{I}\left( \left\|\ell_i - x \right\| \leq t \right) \, \exp\left\{-\mu(B(\ell_i,t))\right\} \, \mathrm{d}t + \sum_{i=1}^{n_0} \, \int_0^{u} \mu(B(\ell_i,t)) \, \exp\left\{-\mu(B(\ell_i,t))\right\} \, \mathrm{d}t. \end{equation} As in previous examples, notice that the second term in the influence function is a constant, that is, it does not depend on $x$. { Moreover, due to the properties of the exponential function \( \exp(-x) \) for \( 0 \leq x \leq 1 \), we can establish that \( J \) satisfies the \( L \)-smoothness assumption without imposing any additional conditions.}

{ There is some similarity of this problem to that of positioning emergency service vehicles in a city, e.g., there are similarities with the discrete optimization formulation given in \cite{daskin83}. A key difference is that in the emergency service setting, one seeks deterministic locations at which to station vehicles, so the target measure is atomic. Here we relax the atomic requirement. The present formulation appears to be more applicable to certain problems in so-called community first responder schemes, wherein one attempts to recruit volunteers across a city so as to minimize a community response time to an out-of-hospital cardiac arrest \cite{berhenjagli21}. The present approach avoids the need for discretization that was used in that paper.}
\demo

\subsection{Neural Networks with a Single Hidden Layer.}\label{sec:networks} Consider functions of the form $\hat{y}(x;\theta)=\frac{1}{N}\sum_{i=1}^N\sigma_*(x;\theta_i)$ where $N$ represents the number of hidden units, $\sigma_*:\mathbb{R}^d \times \mathbb{R}^D$ is an activation function, and $\theta_i \in \mathbb{R}^D$. The population risk is given by: \begin{equation*}
     \mathbb{E}[(y-\hat{y}(x;\theta))^2] = c_0 + \frac{2}{N}\sum_{i=1}^N V(\theta_i) + \frac{1}{N^2}\sum_{i,j=1}^N U(\theta_i,\theta_j)\,
\end{equation*} where $c_0=\mathbb{E}[y^2]$, $V(\theta)=-\mathbb{E}[y\sigma_*(x;\theta)]$, and $U(\theta_1,\theta_2)=\mathbb{E}[\sigma_*(x;\theta_1)\sigma_*(x;\theta_2)]$. It's worth noting that $U(\cdot,\cdot)$ takes on a symmetric positive semidefinite form. For large $N$, Mei et al.~\cite{2018mei} proposed replacing the empirical distribution $\frac{1}{N}\sum_{i=1}^N\delta_{\theta_i}$ with $\mu \in \mathcal{P}(\mathbb{R}^D)$ to approximate the population risk of two-layer neural networks, reformulating the problem as follows: \begin{align*}\mbox{min.} & \quad J(\mu) = c_0 + \int V(\theta)\, \mu(\mathrm{d}\theta) + \frac{1}{2}\int U(\theta_1,\theta_2)\, \mu(\mathrm{d}\theta_1) \, \mu(\mathrm{d}\theta_2) \\ \mbox{s.t.} & \quad  \mu \in \mathcal{P}(\mathcal{X}). \end{align*} 
After some computation, the influence function is derived as: \begin{equation*} h_{\mu}(\theta) = V(\theta) + \int U(\theta,\theta')\, \mu(\mathrm{d}\theta') + c,\end{equation*} where $c$ is a constant in $\mathbb{R}$. {Similar to the previous examples, the linear structure of the influence function implies that a sufficient condition for the \( L \)-smoothness assumption to hold is the boundedness of \( V \) and \( U \).} \demo

\subsection{Cumulative Residual Entropy Maximization.} Consider \begin{align}\label{cre} \mbox{min.} & \quad J(\mu) := \int_{0}^{\infty} \mu\left((\lambda,\infty)\right) \, \log \mu\left((\lambda,\infty)\right) \mathrm{d}\lambda   \nonumber \\ \mbox{s.t.} & \quad \mu \in \mathcal{P}(\mathcal{X}), \end{align}
where $\mathcal{X} := [a,b]$ is a compact interval of $\mathbb{R}^+$. The quantity $-J(\mu)$ is called the \emph{cumulative residual entropy} (CRE) associated with the measure $\mu$~\cite{2004raoetal}. By comparison, when $\mu$ has a density $g_{\mu}$ on $\mathcal{X}$, it is known that the usual \emph{differential entropy}
$$H(\mu) := -\int_{\mathcal{X}} g_{\mu}(x) \log g_{\mu}(x) \mathrm{d}x.$$  There exists a function $\phi$ such that $H(\phi(\mu))$ is related to CRE as $$H(\phi(\mu)) = \frac{-J(\mu)}{\mathbb{E}[X_{\mu}]} - \frac{1}{\mathbb{E}[X_{\mu}]}\, \log \frac{1}{\mathbb{E}[X_{\mu}]},$$ where $\mathbb{E}[X_{\mu}] = \int_{\mathcal{X}} x\mathrm{d}\mu.$ From the chain rule, we can obtain the influence function of $J$ at $\mu$: \begin{equation}\label{creIF}
    h_{\mu}(x) =  \int_{0}^{\infty}  \left(1+\log \left(\mu((\lambda,\infty))\right) \right) \left(\mathbb{I}_{(\lambda,\infty)}(x)-\mu\left((\lambda,\infty)\right) \right) \mathrm{d}\lambda.
\end{equation}
Hence, conclude that $\mu^*=\frac{1}{2} \delta_{a} + \frac{1}{2}\delta_{b}$ {when the base of the logarithm is 2. To see why, for \(\mu^* = \frac{1}{2} \delta_a + \frac{1}{2} \delta_b\),
\[
    \mu^*\left((\lambda, \infty)\right) =
    \begin{cases}
        1, & 0 \leq \lambda < a, \\
        \frac{1}{2}, & a \leq \lambda < b, \\
        0, & \lambda \geq b.
    \end{cases}
\]
For all \(x \in [a, b]\), the influence function becomes
\begin{align}
    h_{\mu^*}(x) &= \int_{[a, x)} \left(1 + \log_2 \frac{1}{2}\right) \left(1 - \frac{1}{2}\right) \mathrm{d}\lambda 
    - \int_{[x, b)} \left(1 + \log_2 \frac{1}{2}\right) \frac{1}{2} \mathrm{d}\lambda \nonumber \\
    &= 0.
\end{align}
This implies that \(h_{\mu^*}(x) = 0\) for all \(x \in [a, b]\), proving that \(\mu^*\) is optimal.} {It is important to note that the \( L \)-smoothness assumption does not hold in this example, as the influence function can be unbounded—for instance, near \( \mu = \delta_a \).} \demo

\subsection{Gaussian Deconvolution}\label{sec:deconvolution} Consider the Gaussian deconvolution model defined by \begin{equation}\label{eq:deconvolution}
    Y_i = W_i + Z_i, \quad i=1,\dots,n.
\end{equation} Here, $Y_1,\dots,Y_n$ represent corrupted observations, and the errors $Z_1,\dots,Z_n$ are independent of $W_1,W_2,\dots,W_n$. In this model, the unknown distribution of $W_i$, denoted as $\nu$ and supported on $\mathcal{X}$, is to be estimated, with $Z_i \sim N(0,\sigma^2)$, where the variance $\sigma^2$ is known. The task is to estimate $\nu$ based on the observed data $Y_1,\dots,Y_n$. The maximum-likelihood estimator (MLE) for $\nu$ is given by \begin{equation}
    \hat{\nu} = \argmax_{\mu \in \mathcal{P}(\mathcal{X})} \sum_{i=1}^n \log \left( \phi_\sigma\,*\, \mathrm{d}\mu(Y_i)\right) \quad where \quad \phi_\sigma\,*\, \mathrm{d}\mu(Y_i)=\int_{\mathcal{X}} \phi_\sigma(Y_i-t) \, \mathrm{d}\mu(t).
\end{equation} {Here $\phi_\sigma$ is the density of $Z_i$}. Therefore, the corresponding optimization problem is given by\begin{align*}\mbox{min.} & \quad J(\mu) = -\sum_{i=1}^n \log \left( \int_{\mathcal{X}} \phi_\sigma(Y_i-t) \, \mathrm{d}\mu(t)\right) \\ \mbox{s.t.} & \quad  \mu \in \mathcal{P}(\mathcal{X}). \end{align*} Then, the influence function of $J$ at $\mu$ is 
$h_{\mu}(x) = n-\sum_{i=1}^n \frac{\phi_\sigma(Y_i-x)}{\int_{\mathcal{X}} \phi_\sigma(Y_i-t) \, \mathrm{d}\mu(t)}$. {Furthermore, since the density function of the Gaussian distribution is bounded on \( \mathcal{X} \), we can prove that the \( L \)-smoothness assumption holds without requiring any additional conditions.} \hfill \demo

\section{DETERMINISTIC FW RECURSION.}\label{sec:dfw}

Recall that our problem of interest is  \begin{align}\label{restrictedoptagain} \mbox{min.} & \quad J(\mu) \nonumber \\ \mbox{s.t.} & \quad \mu \in \mathcal{P}(\mathcal{X}), \tag{$P$}\end{align} where $\mathcal{X}$ is a compact subset of $\mathbb{R}^d$, and $\mathcal{P}(\mathcal{X})$ is the probability space on $\mathcal{X}$, that is, the space of non-negative Borel measures $\mu$ supported on $\mathcal{X}$ such that $\mu(\mathcal{X}) = 1$. In this section, as a method to solve~\eqref{restrictedoptagain}, we present an analogue of the deterministic Frank-Wolfe (dFW) recursion~\cite{1978dunhar} (sometimes called the \emph{conditional gradient} method~\cite{2015bub}) on the probability space $\mathcal{P}(\mathcal{X})$. 

First recall the essential idea of the Frank-Wolfe recursion in $\mathbb{R}^d$, when we are minimizing a smooth function $f: \mathbb{R}^d \to \mathbb{R}$ over a compact convex set $Z \subset \mathbb{R}^d$. We begin with a feasible solution $y_0$ and proceed iteratively, by minimizing a first-order approximation to $f$ at each step, then taking a step in the direction of the minimizer of the approximation, i.e.,
\begin{equation}\label{detFWRd} y_{k+1} = y_k + \eta_k (s_k - y_k), \quad s_k : = \argmin_{s \in Z} \{f(y_k) + \nabla f(y_k)^T (s - y_k)\}, \quad k \geq 0.
\end{equation}
The recursion \eqref{detFWRd} can be simplified by ignoring constants and rearranging terms to obtain the standard form of Frank-Wolfe: \begin{equation}\label{detFWRd1} y_{k+1} = (1-\eta_k)y_k + \eta_k s_k, \quad s_k : = \argmin_{s \in Z} \{\nabla f(y_k)^T s\}, \quad k \geq 0.\end{equation} The obvious advantage of~\eqref{detFWRd1} is that the sequence $\{y_k, k \geq 0\}$ remains feasible, and that $s_k$ is obtained simply, by minimizing a linear function over the compact convex set $Z$. 

To mimic~\eqref{detFWRd1} in probability spaces, we notice that a first-order approximation to $J(\cdot)$ at $\mu_k$ is $J(u) \approx J(\mu_k) + J'_{\mu_k}(u)$, where $J'_{\mu_k}(u)$ denotes the von Mises derivative at $\mu_k$ in the direction $u - \mu_k$, suggesting the following analogue to~\eqref{detFWRd1}:
\begin{equation}\label{dFW0} \mu_{k+1} = (1-\eta_k)\mu_k + \eta_k\left\{\argmin_{u \in \mathcal{P}(\mathcal{X})} \left\{J(\mu_k) + J'_{\mu_k}(u)\right\}\right\}, \quad k \geq 0.\end{equation}
Towards further simplifying~\eqref{dFW0} toward a ``particle update,'' we observe through the following lemma that at any $\mu \in \mathcal{P}(\mathcal{X})$, the ``direction'' $u$ that minimizes the von Mises derivative $J'_{\mu}(u)$ is simply the Dirac measure concentrated at a point $x^*(\mu)$ that minimizes the influence function $h_{\mu}(\cdot)$ at $\mu$.

\begin{lemma} [Solution to FW Subproblem]\label{lem:dFW} Let $\mu \in \mathcal{P}(\mathcal{X})$ be such that $h_{\mu}(\cdot)$ attains its minimum on $\mathcal{X}$. Then, for fixed $\mu$,
\begin{equation}
\argmin_{u \in \mathcal{P}(\mathcal{X})} J'_{\mu}(u) = \delta_{x^*(\mu)}, \mbox{ where } x^*(\mu) \in \argmin_{x \in \mathcal{X}} \,\, h_{\mu}(x).
\end{equation}
\end{lemma}

\begin{proof}{Proof.}
For each $u \in \mathcal{P}(\mathcal{X})$,
\begin{align*} J'_{\mu}(u) & = \int_{\mathcal{X}} h_{\mu}(x) \, u(dx)  \nonumber \\
& \geq \int_{\mathcal{X}} h_{\mu}(x^*(\mu)) \, u(dx)   \nonumber \\
& = h_{\mu}(x^*(\mu)). \qed
\end{align*}    
\end{proof}

From~\eqref{dFW0} and Lemma~\ref{lem:dFW}, we get the deterministic Frank-Wolfe ``particle update'' recursion on probability spaces:
\begin{equation}\label{dFW} \mu_{k+1} = (1-\eta_k)\mu_k + \eta_k\delta_{x^*(\mu_k)}; \quad x^*(\mu_k) \in \argmin_{x \in \mathcal{X}} \,\, h_{\mu_k}(x). \tag{dFW}
\end{equation} Implicit in the recursion~\eqref{dFW} is that the function $h_{\mu}$ attains its minimum on $\mathcal{X}$. Since $\mathcal{X}$ is compact, this is true if, e.g., $h_{\mu}$ is continuous on $\mathcal{X}$. 

As in optimization over $\mathbb{R}^d$, the smoothness of the objective function $J$ plays a pivotal role in analyzing the convergence rate through a ``smooth function inequality'' for probability spaces. In obtaining such an inequality, we need a notion of smoothness of $J$ through an appropriate metric on $\mathcal{P}(\mathcal{X})$ such as the \emph{total variation} distance, as defined through~\eqref{Lsmooth}.

\begin{lemma}[Smooth Functional Inequality] Suppose $J$ is convex and $L$-smooth. Then, for any $\mu,\nu \in \mathcal{P}(\mathcal{X})$, $J$ satisfies \begin{equation}\label{descentlemmagen} 0 \leq J(\nu) - \left( J(\mu) + J'_{\mu}(\nu) \right)  \leq  \frac{L}{2}\|\nu-\mu\|^2. \end{equation}
\end{lemma}
\begin{proof}{Proof.}
    Let $\nu_t = \mu + t(\nu-\mu)$ and {$F(t) = J(\nu_t)$,} where $0\leq t \leq 1$. {Notice that $F(1) = J(\nu)$ and $F(0) = J(\mu)$.} By the fundamental theorem of calculus, we express {\[
J(\nu) - J(\mu) = F(1) - F(0) = \int_0^1 F'(t) \, \mathrm{d}t.
\]  Now, observe that  
\[
F'(t) = \lim_{h \to 0} \frac{F(t+h) - F(t)}{h} = \lim_{h \to 0} \frac{J(\nu_{t} + h(\nu - \mu)) - J(\nu_t)}{h} = J'_{\nu_t}(\nu - \mu),
\]  
where the second equality follows from $\nu_{t+h} = \nu_t + h(\nu - \mu)$, and the last equality is due to the definition of the von Mises derivative. Thus, we can write \[
J(\nu) = J(\mu) + \int_0^1 J'_{\nu_t}(\nu - \mu) \, \mathrm{d}t = J(\mu) + J'_{\mu}(\nu) + \int_0^1 \left(J'_{\nu_t}(\nu - \mu) - J'_{\mu}(\nu)\right) \, \mathrm{d}t.
\] By definition the difference term can be expressed as  
\begin{align}\label{eq:diff}
J'_{\nu_t}(\nu - \mu) - J'_{\mu}(\nu ) &= \int h_{\nu_t}(x) \, \mathrm{d}(\nu - \mu)(x) - \int h_{\mu}(x) \, \mathrm{d}(\nu - \mu)(x) \nonumber \\ &= \int \left(h_{\nu_t}(x) - h_{\mu}(x)\right) \, \mathrm{d}(\nu - \mu)(x).
\end{align}}
Considering the convexity of $J$, we have \begin{equation} \label{convJ} J(\nu) \geq J(\mu) + J'_{\mu}(\nu).\end{equation} {According to~\eqref{eq:diff}}, \begin{align} \label{intineq}\left | \int_0^1 \left(  {J'_{\nu_t}(\nu - \mu)}- J'_{\mu}(\nu) \right) \, \mbox{d}t \right | & \leq \int_0^1 \left |  \int_{\mathcal{X}} (h_{\nu_t}(x)-h_{\mu}(x)) \, (\nu-\mu)( \mathrm{d}x) \right| \, \mbox{d}t \nonumber \\ & \leq \int_0^1 \sup_{x \in \mathcal{X}} |h_{\nu_t}(x)-h_{\mu}(x) | \, \left\|\nu - \mu \right\| \, \mbox{d}t \nonumber \\ & \leq \frac{L}{2} \, \left\|\nu - \mu \right\|^2,
\end{align} where the second inequality follows from Hölder's inequality, and the third inequality results from the L-smoothness. Use~\eqref{convJ} and~\eqref{intineq} to see that the assertion of the lemma holds. \hfill $\blacksquare$
\end{proof}

We next characterize the complexity (in objective function value) of the iterates $(\mu_k, k \geq 1)$ generated by \eqref{dFW}. 

\begin{theorem}[dFW Complexity]\label{thm:dFWcomp}
Suppose $J$ is convex and $L$-smooth, and the step-sizes $\{\eta_k, k \geq 0\}$ in~\eqref{dFW} are chosen as $\eta_k = \frac{2}{k+2}.$ Then, $$J(\mu_k) - J^* \leq \frac{2LR^2}{k+2}, \quad k \geq 1$$ where $J^* := \inf\left\{J(\mu): \mu \in \mathcal{P}(\mathcal{X})\right\}$ and $R:= \sup\left\{\|\mu-\nu\|: \mu,\nu \in \mathcal{P}(\mathcal{X})\right\} \leq 2.$
\end{theorem}

\begin{proof}{Proof.}
    We can write
    \begin{align}
        J(\mu_{k+1}) - J(\mu_k) &\leq J'_{\mu_k}(\mu_{k+1}) + \frac{1}{2}L\|\mu_{k+1}-\mu_k\|^2 \tag{from \eqref{descentlemmagen}} \\ & = \eta_k J'_{\mu_k}(\delta_{x^*(\mu_k)}) + \frac{1}{2}\eta_k^2L\|\delta_{x^*(\mu_k)} - \mu_k\|^2 \nonumber \\
        & \leq \eta_k J'_{\mu_k}(\mu^*) + \frac{1}{2}\eta_k^2L\|\delta_{x^*(\mu_k)} - \mu_k\|^2  \tag{from Lemma~\ref{lem:dFW}} \nonumber \\
        & \leq \eta_k J'_{\mu_k}(\mu^*) + \frac{1}{2}\eta_k^2LR^2 \nonumber \\
        & \leq \eta_k (J^* - J(\mu_k)) + \frac{1}{2}\eta_k^2LR^2. \tag{from convexity}
    \end{align}
Setting $\Delta_k := J(\mu_k) - J^*,$ the above implies that \begin{equation}
\Delta_{k+1} \leq (1-\eta_k)\Delta_k + \frac{1}{2}\eta_k^2LR^2, \quad k \geq 0.
\end{equation}
A simple induction using the fact that $\eta_k = 2/(k+2)$ finishes the proof. \hfill $\blacksquare$
\end{proof}

\begin{algorithm}
\caption{Fully-corrective Frank Wolfe on probability spaces}\label{alg:DFW}
\begin{tabular} {l}
  
  \textbf{Input:} Initial measure $\mu_0\in \mathcal{P}(\mathcal{X})$ \\ 
  \textbf{Output:} Iterates $\mu_1, \dots, \mu_{K}\in \mathcal{P}(\mathcal{X})$\\
  1 $S_0 \leftarrow \{\mu_0\}$\\
  2 \textbf{for} $k=1,2,\dots, K$ \textbf{do}\\
  3 ~~$x^*(\mu_k)\leftarrow \argmin_{x\in \mathcal{X}} h_{\mu_k}(x)$\\
  4 ~~$S_{k+1}\leftarrow S_k\cup\{\delta_{x^*(\mu_k)}\}$\\
  5 ~~$\mu_{k+1}\leftarrow \argmin_{\mu\in \mathrm{conv}(S_{k+1})} J(\mu)$\\
  6 \textbf{end for}\\
   
 \end{tabular}
\end{algorithm} Two further discussion points about~\eqref{dFW} and its properties are noteworthy.
\begin{enumerate}
\item[(a)] First, the~\eqref{dFW} recursion solves an infinite-dimensional optimization problem by accumulating point masses located strategically in $\mathbb{R}^d$. This is remarkable because an infinite-dimensional problem is being solved without explicit finite dimensionalization operations such as gridding.  Although, the computational price manifests in a different form, since constructing each iterate involves solving a \emph{global optimization} problem over the compact set $\mathcal{X} \subset \mathbb{R}^d$. This is a formidable task in principle, but as~\cite{2017boygeorec} notes, and as we have observed elsewhere when solving an emergency response problem \cite{berhenjagli21}, there is often a lot of structure in specific contexts that allows for solving the global optimization problems efficiently. Such structure can be combined with imprecise solving at each step, an idea we pursue in the next section.  \item[(b)] There is evidence~\cite{2013brepik,2017boygeorec} that during implementation, a more nuanced version of Frank-Wolfe, called the \emph{fully corrective} version, performs better. As seen in Algorithm~\ref{alg:DFW}, the simple modification in fully corrective Frank-Wolfe is easily internalized. Recall that when using regular Frank-Wolfe leading to~\eqref{dFW}, $\mu_{k+1}$ is obtained as a convex combination of the previous iterate $\mu_k$ and the minimizer $\delta_{x^*(\mu_k)}$ of $h_{\mu_k}$. In the fully corrective version, however, $\mu_{k+1}$ is obtained as the minimum of $J$ over the convex hull of $\delta_{x^*(\mu_k)}, j=1,2,\ldots,k.$ {Furthermore, this fully corrective step, as shown in Step 5, is equivalent to solving the following optimization problem:
\begin{equation}
\min_{p_1, \dots, p_k\in \mathbb{R}} J\left( \sum_{i=1}^{k} p_i \delta_{x^*(\mu_i)} \right) 
\quad \text{s.t.} \quad \sum_{i=1}^{k} p_i = 1, \quad p_i \geq 0.
\end{equation}
Since \( J \) is convex, this results in a finite dimensional convex optimization problem, which remains computationally feasible in practice.}
\end{enumerate}

\section{STOCHASTIC FW RECURSION.}\label{sec:sfw}

We now consider the often-encountered scenario where the influence function $h_{\mu}$ associated with the objective $J$ is not directly observable but we have access to unbiased Monte Carlo observations through a first-order oracle. Precisely, suppose that $Y$ is a random variable defined on { a probability space $(\mathcal{X},\mathcal{A},P)$}, and that $F(\cdot,Y): \mathcal{P}(\mathcal{X}) \to \mathbb{R}$, $H_{\mu}(\cdot,Y) : \mathcal{X} \to \mathbb{R}$ are random functions that form unbiased estimators of $J$ and $h_{\mu}$, respectively, that is, $\mathbb{E}[F(\cdot,Y)] = J(\cdot), \mathbb{E}[H_{\mu}(\cdot,Y)] = h_{\mu}(\cdot).$ Suppose that $F(\cdot,Y),H_{\mu}(\cdot,Y)$ are observable (only) using Monte Carlo so that we can define the sample-average estimators 
\begin{equation}\label{unbiased_inf}
J_{m}(\mu) := \frac{1}{m}\sum_{j=1}^m F(\mu,Y_j); \quad H_{\mu,m}(x) := \frac{1}{m}\sum_{j=1}^m H_{\mu}(x,Y_j), \quad \mu \in \mathcal{P}(\mathcal{X}), x \in \mathcal{X}.
\end{equation}
For further intuition on $J_m$ and $H_{\mu,m}$, consider the $P$-means example discussed in Section~\ref{sec:pmeans} where we saw that
\begin{align}\label{totcost_pmeansagain} J(\mu) &= \sum_{i=1}^{n_0} \, \int_0^{u}  \exp\left\{-\mu(B(\ell_i,t))\right\} \, \mathrm{d}t, \quad \mu \in \mathcal{P}(\mathcal{X}).
\end{align}
and that \begin{equation}\label{inf_pmeansagain} h_{\mu}(x) = -\sum_{i=1}^{n_0} \, \int_0^{u} \mathbb{I}\left( \left\|\ell_i - x \right\| \leq t \right) \, \exp\left\{-\mu(B(\ell_i,t))\right\} \, \mathrm{d}t + \sum_{i=1}^{n_0} \, \int_0^{u} \mu(B(\ell_i,t)) \, \exp\left\{-\mu(B(\ell_i,t))\right\} \, \mathrm{d}t. \end{equation} Unbiased estimators $J_m, H_{\mu,m}$ in~\eqref{unbiased_inf} for $J$ and $h_{\mu}$, respectively, can then be constructed using 
\begin{align} F(\mu,Y_j) & = \sum_{i=1}^{n_0} u \, \exp\{-\mu(B(\ell_i,Y_j))\}; \mbox{ and } \notag\\
H_{\mu}(x,Y_j) & = \sum_{i=1}^{n_0} u\,\left[-\mathbb{I}\left( \left\|\ell_i - x \right\| \leq Y_j \right)+ \mu(B(\ell_i,Y_j)) \right] \, \exp\left\{-\mu(B(\ell_i,Y_i))\right\}, \label{inf_pmeans_est}  
\end{align}
where $Y_j, j = 1,2,\ldots,n$ are iid copies of $Y \sim \mbox{Uniform}(0,u)$. 


The existence of an unbiased Monte Carlo estimator for $h_{\mu}(\cdot)$ motivates the following stochastic Frank-Wolfe (sFW) recursion. (A fully corrective version appears as Algorithm~\ref{alg:fullycorrectivestoch}):
\begin{align}\label{sfw} \mu_{k+1} &= (1-\eta_k)\mu_k + \eta_k\delta_{\hat{x}_{k+1}(m_{k+1})} \tag{sFW} \\
\hat{x}_{k+1}(m_{k+1}) & \in \argmin_{x \in \mathcal{X}} \,\, \left\{H_{\mu_k,m_{k+1}}(x)\right\}. \nonumber
\end{align} { Here, \( m_k \) represents the number of samples at the \( k \)th iteration.}

\begin{algorithm}
\caption{Fully-corrective stochastic Frank Wolfe on probability spaces}\label{alg:fullycorrectivestoch}
\begin{tabular} {l}
  
  \textbf{Input:} Initial measure $\mu_0\in \mathcal{P}(\mathcal{X})$, parameter $c$ \\ 
  \textbf{Output:} Iterates $\mu_1, \dots, \mu_{K}\in \mathcal{P}(\mathcal{X})$\\
  1 $S_0 \leftarrow \{\mu_0\}$\\
  2 \textbf{for} $k=1,2,\dots, K$ \textbf{do}\\
  3 ~~$m_{k+1}\leftarrow c\left(k+2\right)^2$ \\
  4 ~~$\hat{x}_{k+1}(m_{k+1})\leftarrow \argmin_{x\in \mathcal{X}} H_{\mu_k,m_{k+1}}(x)$\\
  5 ~~$S_{k+1}\leftarrow S_k\cup\{\delta_{\hat{x}_{k+1}(m_{k+1})}\}$\\
  6 ~~$\mu_{k+1}\leftarrow \argmin_{\mu\in \mathrm{conv}(S_{k+1})} J(\mu)$\\
  7 \textbf{end for}\\
   
 \end{tabular}
\end{algorithm}

In writing~\eqref{sfw}, we are implicitly assuming that $H_{\mu_k,m_{k+1}}$ attains its minimum on $\mathcal{X}$. (We can suitably modify Lemma~\ref{lem:infcont} to obtain sufficient conditions for the continuity of $H_{\mu_k,m_{k+1}}$ on $\mathcal{X}$.) Also, it is important that even though $H_{\mu_k,m_{k+1}}$ is an unbiased estimator of $h_{\mu_k}$, $\hat{x}_{k+1}(m_{k+1})$ is not, in general, an unbiased estimator of $\arg\inf_{x \in \mathcal{X}} h_{\mu_k}(x)$. However, $\hat{x}_{\mu,m}$ is a consistent estimator of $\arg\inf_{x \in \mathcal{X}} h_{\mu}(x)$ under certain regularity conditions (see for instance~\cite{2009shadenrus}) suggesting that increasing $m_k \to \infty$ as $k \to \infty$ will result in some form of consistency. 
In the following theorem, convergence (in function value) along with a complexity bound on the sequence $\{J(\mu_k), k \geq 1\}$ is attained by ``killing'' the bias due to $H_{\mu_k,m_k}$ through a sample size increase. The proof is not novel, and follows along lines similar to what is available in the Euclidean context~\cite{2018botcurnoc}.

\begin{theorem}[Complexity]\label{thm:sfwcomplexity}
Suppose that $J$ is convex and $L$-smooth, that
\begin{equation}\label{comp_thm_step} \eta_k = \frac{2}{k+2}; \quad m_k \ge \left(\frac{c_0(k+2)}{LR}\right)^2
\end{equation}
and the CLT-scaling assumption holds, that is, there exists $c_0 < \infty$ such that for all $\mu \in \mathcal{P}(\mathcal{X})$,
\begin{equation}\label{clt-sc} \mathbb{E}\left[ \sqrt{m}\left \| H_{\mu,m} - h_{\mu} \right\|_{\infty} \right] \leq c_0. \tag{CLT-sc}
\end{equation}
Then, the iterates $\mu_k, k \geq 1$ generated by the~\eqref{sfw} recursion satisfy $$\mathbb{E}\left[ J(\mu_k) - J(\mu^*)\right] \leq \frac{4LR^2}{k+2}, \quad k \geq 1$$ where $R = \sup\{\|\mu_1-\mu_2\|, \mu_1,\mu_2 \in \mathcal{P}(\mathcal{X})\}.$
\end{theorem}  
\begin{proof}{Proof.}
Write
\begin{align*}
J(\mu_{k+1}) & \leq J(\mu_k) + J'_{\mu_k}(\mu_{k+1}) + \frac{L}{2}\|\mu_{k+1}-\mu_k\|^2  \tag{smooth} \\
&=  J(\mu_k) + \eta_k J'_{\mu_k}(\delta_{\hat{x}_{k+1}(m_{k+1})} ) + \frac{L}{2}\eta_k^2\|\delta_{\hat{x}_{k+1}(m_{k+1})}-\mu_k\|^2 \\
&\leq J(\mu_k) + \eta_k \int_{\mathcal{X}} h_{\mu_k} \,  \mbox{d}(\delta_{\hat{x}_{k+1}(m_{k+1})}-\mu_k) + \frac{L}{2}\eta_k^2 R^2  \tag{$\|\delta_{\hat{x}_{k+1}(m_{k+1})}-\mu_k\| \leq R$} \\
&\leq J(\mu_k) + \eta_k \int_{\mathcal{X}} H_{\mu_k,m_{k+1}} \,  \mbox{d}(\mu^*-\mu_k) \tag{by optimality of $\hat{x}_{k+1}(m_{k+1})$} \\
& \qquad + \eta_k \int_{\mathcal{X}} (h_{\mu_k}-H_{\mu_k,m_{k+1}}) \,  \mbox{d}(\delta_{\hat{x}_{k+1}(m_{k+1})}-\mu_k) + \frac{L}{2}\eta_k^2 R^2 \\
& = J(\mu_k) + \eta_k \int_{\mathcal{X}} h_{\mu_k} \,  \mbox{d}(\mu^*-\mu_k) + \eta_k \int_{\mathcal{X}} (h_{\mu_k}-H_{\mu_k,m_{k+1}}) \,  \mbox{d}(\delta_{\hat{x}_{k+1}(m_{k+1})}-\mu^*) + \frac{L}{2}\eta_k^2 R^2 \\
& \leq J(\mu_k) + \eta_k (J^*-J(\mu_k)) + \eta_k \int_{\mathcal{X}} (h_{\mu_k}-H_{\mu_k,m_{k+1}}) \,  \mbox{d}(\delta_{\hat{x}_{k+1}(m_{k+1})}-\mu^*) + \frac{L}{2}\eta_k^2 R^2. \tag{convexity}
\end{align*}
Conditioning both sides on $\mathcal{F}_k$, taking expectation, and denoting $\Delta_k := J(\mu_k) - J^*$, we get \begin{align}\label{condexpformart}
\mathbb{E}\left[\Delta_{k+1} \, \vert \, \mathcal{F}_k \right] & \leq (1-\eta_k) \Delta_k + \eta_k \mathbb{E}\left[\left.\int_{\mathcal{X}} (h_{\mu_k}-H_{\mu_k,m_{k+1}}) \,  \mbox{d}(\delta_{\hat{x}_{k+1}(m_{k+1})}-\mu^*) \, \right\vert \, \mathcal{F}_k \right] + \frac{L}{2}\eta_k^2 R^2 \nonumber \\ & \leq (1-\eta_k)\Delta_k + R\eta_k \mathbb{E}\left[\| h_{\mu_k} - H_{\mu_k,m_{k+1}} \|_{\infty} \, \vert \, \mathcal{F}_k \right] + \frac{L}{2}\eta_k^2R^2 \nonumber \\ & \leq (1-\eta_k)\Delta_k + R\eta_k \frac{c_0}{\sqrt{m_{k+1}}} + \frac{L}{2}\eta_k^2R^2 \nonumber \\
& \le (1-\eta_k)\Delta_k + L\eta_k^2R^2. \end{align}
Taking expectations again, we get
\begin{align} \mathbb{E}\left[\Delta_{k+1}\right] \leq  (1-\eta_k)\mathbb{E}\left[\Delta_k\right] + L\eta_k^2R^2, \quad k \geq 0.\end{align}
Now use induction to conclude that the assertion holds. \hfill $\blacksquare$\end{proof} 

Apart from the stipulations on the step size and sample size appearing in~\eqref{comp_thm_step}, Theorem~\ref{thm:sfwcomplexity} requires that the CLT-scaling assumption in~\eqref{clt-sc} is satisfied. The CLT-scaling assumption in~\eqref{clt-sc} is essentially a stipulation that the sample-paths $H_{\mu,m}(\cdot) - h_{\mu}(\cdot)$ do not exhibit excessive fluctuations, as is sometimes codified through requirements on the modulus of continuity~\cite[p. 80]{1999bil}. CLT-scaling appears to hold in many settings. For instance, consider again the $P$-means example discussed in Section~\ref{sec:pmeans}. Applying Theorem 6.1 in~\cite{geer2000empirical}, we can show that the empirical process $\{\sqrt{m}(H_{\mu,m}(x) - h_{\mu}(x)), x \in \mathcal{X}\}$ is a P-Donsker class~\cite[p. 88]{geer2000empirical}, implying (using the continuous mapping theorem) that $\|\sqrt{m}(H_{\mu,m} - h_{\mu})\|_{\infty} \inD \|Z\|_{\infty}$, where $Z = \{Z(x),x\in \mathcal{X}\}$ is a zero-mean Gaussian process indexed by $x$. Furthermore, since it can also be shown that $\|\sqrt{m}(H_{\mu,m} - h_{\mu})\|_{\infty}, m \geq 1$ is uniformly integrable, we see that $\mathbb{E}\left[ \| \sqrt{m}(H_{\mu,m} - h_{\mu}) \|_{\infty}\right]  \to \mathbb{E}\left[\|Z\|_{\infty}\right] < \infty$, implying that the CLT-scaling assumption holds for the $P$-means example.

It turns out that the same postulates as Theorem~\ref{thm:sfwcomplexity} also guarantee almost sure consistency on the optimality gap sequence $\{J(\mu_k)-J^*, k \geq 1\}$, and on the sequence of measures $\{\mu_k, k \geq 1\}$ under the weak topology.

\begin{theorem}[Almost Sure Convergence Rate]\label{thm:sfwasconsistency}
Suppose the postulates of Theorem~\ref{thm:sfwcomplexity} hold. Then, the iterates $\mu_k, k \geq 1$ generated by the~\eqref{sfw} recursion satisfy $$k^{1-\delta}(J(\mu_k)-J^*) \,\, \as \,\, 0, \quad \forall \, 0<\delta<1.$$  Moreover, { if the minimizer \(\mu^* := \arg\inf_{\mu \in \mathcal{P}(\mathcal{X})} J(\mu)\) is unique in the weak topology, then the sequence \((\mu_k)_{k \geq 1}\) converges to \(\mu^*\) almost surely in the weak topology.}
\end{theorem} 
\begin{proof}{Proof.}
Define, for $k \ge 1$,
$$M_k = k^{1-\delta}\Delta_{k} + \sum_{j=k}^{\infty} \frac{4LR^2}{(j+1)^{1+\delta}}.$$
Then \eqref{condexpformart} implies that $(M_k: k \ge 1)$ is a non-negative supermartingale, since
\begin{align}\label{supermart}
\mathbb{E}\left[\left.(k+1)^{1-\delta}\Delta_{k+1} + \sum_{j=k+1}^{\infty} \frac{4LR^2}{(j+1)^{1+\delta}} \right\vert \, \mathcal{F}_k  \right]  &\leq (k+1)^{1-\delta}\left((1-\eta_k)\Delta_k + LR^2\eta^2_k\right) + \sum_{j=k+1}^{\infty} \frac{4LR^2}{(j+1)^{1+\delta}} \nonumber\\
&\le \frac{k(k+1)^{1-\delta}}{k+2}\Delta_k + \sum_{j=k}^{\infty} \frac{4LR^2}{(j+1)^{1+\delta}} \nonumber\\
&\leq k^{1-\delta}\Delta_{k} + \sum_{j=k}^{\infty} \frac{4LR^2}{(j+1)^{1+\delta}}.
\end{align}
By applying the martingale convergence theorem~\cite{durrett2019probability}, we deduce the existence of a non-negative random variable $X$ such that
\begin{equation} k^{1-\delta}\Delta_{k} + \sum_{j=k}^{\infty} \frac{4LR^2}{(j+1)^{1+\delta}} \,\, \as \,\, X,
\end{equation}
and that
\begin{equation}\label{expbdmart} \mathbb{E}\left[k^{1-\delta}\Delta_{k} + \sum_{j=k}^{\infty} \frac{4LR^2}{(j+1)^{1+\delta}}\right] \geq \mathbb{E}[X], \quad k \geq 1.
\end{equation} 
Moreover, since $\mathbb{E}\left[k^{1-\delta}\Delta_{k} + \sum_{j=k}^{\infty} \frac{4LR^2}{(j+1)^{1+\delta}}\right] \to 0$ as $k \to \infty$ (from Theorem~\ref{thm:sfwcomplexity}),~\eqref{expbdmart} guarantees that $\mathbb{E}[X] \leq 0.$ Consequently, we see that $X=0$ with probability one and then, as $k \to \infty$,
\begin{equation}\label{} k^{1-\delta}\Delta_{k} \,\, \as \,\, 0. \end{equation}
Finally, referring to Section 3 in~\cite{1978dunhar}, we establish that the sequence $(\mu_k, k \geq 1)$ converges to $\mu^* =\arg\inf_{\mu \in \mathcal{P}(\mathcal{X})} J(\mu)$ in the weak topology. {Define the level set \(L(\frac{1}{n}) := \{\mu \in \mathcal{P}(\mathcal{X}) \,|\, J(\mu) \leq J^* + \frac{1}{n}\}\), where \(J^* = \inf_{\mu \in \mathcal{P}(\mathcal{X})} J(\mu)\). Since \(J(\mu_k) \,\, \as \,\, J^*\) as \(k \to \infty\) and the convexity of \(J\) ensures its lower semicontinuity in the weak topology, we can construct a strictly increasing sequence \(\{k_n, n \geq 1\}\) such that for each \(n\),
\begin{equation}
    \mu_k \in L\left(1/n\right), \quad \forall k \geq k_n,
\end{equation}
almost surely. Using this, we define a nested sequence of neighborhoods \(N_k\) by setting \(N_k = L(\frac{1}{n})\) for \(k_n \leq k < k_{n+1}\). Consequently, \(\mu_k \in N_k\) for all \(k \geq k_1\), and the sequence of neighborhoods satisfies \(N_k \downarrow \{\mu^*\}\) monotonically. It follows that \(\mu_k \as \mu^*\) in the weak topology.} 
\hfill
 $\blacksquare$\end{proof}


The following straightforward corollary is intended to provide insight when, in practice, a fixed-step method is used and the subproblems are solved inexactly.

\begin{corollary}[Fixed-Step Fixed-Sample Inexact SFW]\label{thm:sfwfixedstep} Suppose that $J$ is convex and $L$-smooth. Consider the fixed-step fixed-sample inexact stochastic Frank-Wolfe recursions
\begin{align}\label{fixedstepsfw} \mu_{k+1} &= (1-\eta)\mu_k + \eta \delta_{\hat{x}_{k+1}(m)} \\
\hat{x}_{k+1}(m) & \in \left\{x \in \mathcal{X} : H_{\mu_k,m}(x) - \min_{x \in \mathcal{X}}H_{\mu_k,m}(x) \leq \tilde{\epsilon} \right\}. \nonumber
\end{align}
Suppose
\begin{equation}
\quad m \ge \left(\frac{4c_0}{LR\eta}\right)^2; \quad \tilde{\epsilon} \le \frac{LR^2}{4}\eta
\end{equation}
and the CLT-scaling assumption~\eqref{clt-sc} holds. Then, the iterates $\mu_k, k \geq 1$ generated by the~\eqref{sfw} recursion satisfy $$\mathbb{E}\left[ J(\mu_k) - J(\mu^*)\right] \leq (1-\eta)^{k-1}\Delta_1 + \left(1-(1-\eta)^{k-1}\right)LR^2\eta,$$ where $R := \sup\{\|\mu_1-\mu_2\|, \mu_1,\mu_2 \in \mathcal{P}(\mathcal{X})\}.$ \end{corollary} 

\begin{proof}{Proof.}
Notice that
\begin{equation}\label{inexactCond}
    \int_{\mathcal{X}} H_{\mu_k,m}(x) \,  \delta_{\hat{x}_{k+1}(m)} (\mbox{d}x) \leq \min_{x \in \mathcal{X}}H_{\mu_k,m}(x) + \tilde{\epsilon} \leq \int_{\mathcal{X}} H_{\mu_k,m}(x) \,  \mu^* (\mbox{d}x) + \tilde{\epsilon}.
\end{equation}
By following the same procedure as outlined in Theorem~\ref{thm:sfwcomplexity} and substituting~\eqref{inexactCond}, we obtain that \begin{align*}
    J(\mu_{k+1}) &\leq J(\mu_k) + \eta \int_{\mathcal{X}} H_{\mu_k,m} \,  \mbox{d}(\mu^*-\mu_k) + \tilde{\epsilon} \eta + \eta \int_{\mathcal{X}} (h_{\mu_k}-H_{\mu_k,m}) \,  \mbox{d}(\delta_{\hat{x}_{k+1}(m)}-\mu_k) + \frac{L}{2}\eta^2 R^2 \\
    &\leq J(\mu_k) + \eta (J^*-J(\mu_k)) + \tilde{\epsilon} \eta + \eta \int_{\mathcal{X}} (h_{\mu_k}-H_{\mu_k,m}) \,  \mbox{d}(\delta_{\hat{x}_{k+1}(m)}-\mu^*) + \frac{L}{2}\eta^2 R^2.
\end{align*}
Applying the same method as demonstrated in~\eqref{condexpformart}, with a fixed step size $\eta$ for each $k$, we can derive the inequality
\begin{align}  \mathbb{E}\left[\Delta_{k+1} \, \vert \, \mathcal{F}_k \right]
    &\leq (1-\eta)\Delta_k + R\eta \frac{c_0}{\sqrt{m}} + \tilde{\epsilon} \eta + \frac{L}{2}\eta^2R^2.
\end{align}
Taking expectations, $\mathbb{E}[\Delta_{k+1}] \le (1-\eta)\mathbb{E}[\Delta_k] + LR^2\eta^2$ and the result follows by induction. \hfill
 $\blacksquare$\end{proof}

 Notice that Corollary~\ref{thm:sfwfixedstep} implies the scaling relationships {$m = O(\eta^{-2})$} and $\tilde{\epsilon} = O(\eta)$ between the fixed sample size $m$, fixed step size $\eta$, and the tolerance $\tilde{\epsilon}$. Furthermore, and in analogy to results in the Euclidean context~\cite{2018botcurnoc}, Corollary~\ref{thm:sfwfixedstep} implies exponential convergence to the $\tilde{\epsilon}$-ball assuming that the fixed step and fixed sample size are chosen according to the scaling relationships.

 We next state a central limit theorem on the estimated objective function value at the estimated solution. Akin to the Euclidean context, this result could in principle form the basis for statistical inference, and for finite-time algorithmic stopping of the (sFW) recursion. 

\begin{theorem}[Central Limit Theorem]\label{thm:sfwclt} Suppose that the iterates $\mu_k, k \geq 1$ generated by the~\eqref{sfw} recursion satisfy the conditions that $\mathcal{G}:=\left\{ F(\mu, \cdot): \mu \in \mathcal{P}(\mathcal{X}) \right\}$ is a P-Donsker class, and $\| F(\mu_n, \cdot) - F(\mu^*, \cdot) \|_* \, \inP \, 0$, where $\mu^* := \arg\inf_{\mu \in \mathcal{P}(\mathcal{X})} J(\mu)$ is unique. Then $$\sqrt{n}\left( J_n(\mu_n) - J(\mu^*) \right) \,\, \inD \,\, \mathcal{N} (0, \mathbb{E}[F(\mu^*, Y)^2]-\mathbb{E}[F(\mu^*, Y)]^2 ).$$ { Here, for any $g$ in the space $\left\{ g : \mathcal{X} \to \mathbb{R} \, : \, \mathbb{E}[g(Y)^2] < \infty \right\}$, the norm $\|\cdot\|_*$ is defined as $\|g\|_*:=\mathbb{E}[g(Y)^2]^{\frac{1}{2}}$.} \end{theorem}  
\begin{proof}{Proof.} { Suppose the random variable \( Y \) follows the distribution \( Q \). Let \( Q_n \) denote the empirical distribution based on independent samples \( Y_1, \dots, Y_n \), given by $P_n=\frac{1}{n}\sum_{i=1}^n\delta_{Y_i}$. We define the empirical process \( \nu_n \) as: } \begin{equation*}
    \nu_n(g) := \sqrt{n} \int_{\mathcal{X}} g \,  \mbox{d}({Q_n-Q}), \quad g \in \mathcal{G},
\end{equation*} We can write \begin{align}\label{AnBnCn}
    \sqrt{n}\left( J_n(\mu_n) - J(\mu^*) \right) &= \nu_n(F(\mu_n, \cdot)) + \sqrt{n}\left( J(\mu_n) - J(\mu^*) \right) \nonumber\\ &= \nu_n(F(\mu^*, \cdot)) +\left( \nu_n(F(\mu_n, \cdot))-\nu_n(F(\mu^*, \cdot)) \right) + \sqrt{n}\left( J(\mu_n) - J(\mu^*) \right) \nonumber\\ &=: A_n + B_n + C_n.
\end{align} Given that $\mathcal{G}$ is P-Donsker, according to Definition 6.1 in~\cite{geer2000empirical}, we know \begin{equation*}
    A_n \,\, \inD \,\, \mathcal{N} (0, \mathbb{E}[F(\mu^*, Y)^2]-\mathbb{E}[F(\mu^*, Y)]^2 ).
\end{equation*} Based on Theorem~\ref{thm:sfwasconsistency}, we can conclude \begin{equation*}
    C_n \,\, \as \,\, 0.
\end{equation*} Finally, we claim that \begin{equation*}
    B_n \,\, \inP \,\, 0.
\end{equation*} Since $\mathcal{G}$ is P-Donsker, for any $\eta>0$, there exists $\delta>0$ such that \begin{equation*}
    \limsup_{n\rightarrow\infty} \, P \left( \sup_{\|g_1-g_2\|_* \leq \delta} \left| \nu_n(g_1)-\nu_n(g_2) \right| > \eta \right) < \eta.
\end{equation*} Then, define $\tilde{\Omega}_n:= \left\{ \sup_{\|g_1-g_2\|_* \leq \delta} \left| \nu_n(g_1)-\nu_n(g_2) \right| > \eta \right\}$, we have \begin{equation*}
    P \left( \tilde{\Omega}_n \right) < 2 \eta
\end{equation*} eventually. Now, let $\Omega_n:= \left\{ \| F(\mu_n, \cdot) - F(\mu^*, \cdot) \|_* > \delta \right\}$, as $\| F(\mu_n, \cdot) - F(\mu^*, \cdot) \|_* \, \inP \, 0$, we have \begin{equation*}
    P \left( \Omega_n \right) <  \eta
\end{equation*} eventually. Since $\Omega^c_n \cap \tilde{\Omega}^c_n \subset \left\{ \left| \nu_n(F(\mu_n, \cdot))-\nu_n(F(\mu^*, \cdot)) \right| \leq \eta \right\}$ and $P\left( \tilde{\Omega}_n \cup \Omega_n \right) \leq 3 \eta$ eventually, we can derive \begin{equation*}
    P\left(  \left| \nu_n(F(\mu_n, \cdot))-\nu_n(F(\mu^*, \cdot)) \right| < \eta \right) \geq 1 - 3\eta
\end{equation*} eventually, which implies $B_n \,\, \inP \,\, 0$. Consequently, returning to~\eqref{AnBnCn}, we can conclude \begin{equation*}\sqrt{n}\left( J_n(\mu_n) - J(\mu^*) \right) \,\, \inD \,\, \mathcal{N} (0, \mathbb{E}[F(\mu^*, Y)^2]-\mathbb{E}[F(\mu^*, Y)]^2 ). \end{equation*} implying that the assertion holds.
\hfill $\blacksquare$\end{proof}

The two conditions, (i) $\mathcal{G}:=\left\{ F(\mu, \cdot): \mu \in \mathcal{P}(\mathcal{X}) \right\}$ is a P-Donsker class, and (ii) $\| F(\mu_n, \cdot) - F(\mu^*, \cdot) \|_* \, \inP \, 0$, of Theorem~\ref{thm:sfwclt} are routinely met. Consider again the $P$-means problem of Section~\ref{sec:pmeans}, where we have seen that the function $F(\mu, \cdot)$ takes the form $F(\mu, Y)=\sum_{i=1}^{n_0} u \exp \left\{-\mu(B(\ell_i,Y))\right\}.$ Applying Theorem 6.1 in~\cite{geer2000empirical}, we can prove that $\mathcal{G}$ is indeed a P-Donsker class, and furthermore, since $\mu_n$ weakly converges to $\mu^*$, that $\| F(\mu_n, \cdot) - F(\mu^*, \cdot) \|_* \, \inP \, 0$.

\subsection{Handling Nonconvex Objectives}

To analyze the scenario where $J(\cdot)$ is a nonconvex and $L$-smooth function, we introduce the Frank-Wolfe gap in the probability space defined as \begin{equation}\label{FWgap}
    G(\mu):= \max_{\nu \in \mathcal{P}(\mathcal{X})} -J'_{\mu}(\nu),
\end{equation} In Euclidean spaces the Frank-Wolfe gap serves as a crucial criterion for assessing the convergence of Frank-Wolfe methods~\cite{pokutta2023frankwolfe}, particularly in nonconvex settings~\cite{lacostejulien2016convergence,reddi2016stochastic}. In the space of probability measures, $\mu \in \mathcal{P}(\mathcal{X})$ is locally optimal if and only if the Frank-Wolfe gap $G(\mu)=0$. Even when $J$ lacks convexity, the Fixed-Step Fixed-Sample Stochastic Frank-Wolfe method can still be employed, leading to the following result.

\begin{theorem}\label{thm:sfwcomplexNonConvex} Suppose $J$ is $L$-smooth but not necessarily convex, and the CLT-scaling assumption~\eqref{clt-sc} holds. The iterates $\mu_k, k \geq 1$ generated by the~\eqref{sfw} recursion with parameters \begin{equation*} \eta_k = \eta = \sqrt{\frac{2(J(\mu_0)-J(\mu^*))}{L\, R^2 \, T}}; \qquad m_k = m = T \end{equation*} for all $k\in\{0,\dots, T-1 \}$ satisfy \begin{equation}
    \mathbb{E}\left[ G(\mu_a)\right] \leq \frac{R}{\sqrt{T}}\left( c_0 + \sqrt{2L(J(\mu_0)-J(\mu^*))}\right)
\end{equation} where $\mu_a$ is chosen uniformly at random from $\{\mu_k\}_{k=0}^{T-1}$.\end{theorem}
\begin{proof}{Proof.} At each iteration $k$, let $\nu_k \in \argmax_{\nu \in \mathcal{P}(\mathcal{X})} -J'_{\mu_k}(\nu)$, implying $G(\mu_k)=-J'_{\mu_k}(\nu_k)$. Notice that
\begin{align}
    J(\mu_{k+1}) &\leq J(\mu_k) + \eta \int_{\mathcal{X}} H_{\mu_k,m} \,  \mbox{d}(\delta_{\hat{x}_{k+1}(m)}-\mu_k) + \eta \int_{\mathcal{X}} (h_{\mu_k}-H_{\mu_k,m}) \,  \mbox{d}(\delta_{\hat{x}_{k+1}(m)}-\mu_k) + \frac{L}{2}\eta^2 R^2 \nonumber \\
    &\leq J(\mu_k) + \eta \int_{\mathcal{X}} H_{\mu_k,m} \,  \mbox{d}(\nu_k-\mu_k) + \eta \int_{\mathcal{X}} (h_{\mu_k}-H_{\mu_k,m}) \,  \mbox{d}(\delta_{\hat{x}_{k+1}(m)}-\mu_k) + \frac{L}{2}\eta^2 R^2 \nonumber \\
    &= J(\mu_k) + \eta \int_{\mathcal{X}} h_{\mu_k} \,  \mbox{d}(\nu_k-\mu_k) + \eta \int_{\mathcal{X}} (h_{\mu_k}-H_{\mu_k,m}) \,  \mbox{d}(\delta_{\hat{x}_{k+1}(m)}-\nu_k) + \frac{L}{2}\eta^2 R^2 \nonumber \\ &\leq J(\mu_k) - \eta G(\mu_k) + \eta R \| h_{\mu_k} - H_{\mu_k,m} \|_{\infty} + \frac{L}{2}\eta^2 R^2.
\end{align} The first inequality follows from the $L$-smoothness, while the second one follows from the optimality of $\hat{x}_{k+1}(m)$, { i.e., $\hat{x}_{k+1}(m)  \in \argmin_{x \in \mathcal{X}} \,\, H_{\mu_k,m}(x)$}. The last inequality arises from Hölder's inequality. Taking the expectation and utilizing~\eqref{clt-sc}, we obtain \begin{equation*}
      \eta \mathbb{E}\left[G(\mu_k) \right] \leq \mathbb{E}\left[J(\mu_{k}) \right] - \mathbb{E}\left[J(\mu_{k+1}) \right] + \eta R \frac{c_0}{\sqrt{m}} + \frac{L}{2}\eta^2 R^2.
\end{equation*}
Then, summing over $k$
\begin{align*}
    \eta \sum_{k=0}^{T-1} \mathbb{E}\left[G(\mu_k) \right] &\leq \mathbb{E}\left[J(\mu_{0}) \right] - \mathbb{E}\left[J(\mu_{T}) \right] + T \eta R \frac{c_0}{\sqrt{m}} + \frac{L}{2} T \eta^2 R^2 \\ &\leq J(\mu_{0})  - J(\mu^*)  + T \eta R \frac{c_0}{\sqrt{m}} + \frac{L}{2} T \eta^2 R^2.
\end{align*} Therefore, \begin{align*}
    \mathbb{E}\left[G(\mu_a) \right] \leq \frac{J(\mu_{0})  - J(\mu^*)}{T \eta} + \frac{c_0}{\sqrt{m}} R + \frac{1}{2} L R^2 \eta = \frac{R}{\sqrt{T}}\left(c_0 + \sqrt{2L(J(\mu_{0})  - J(\mu^*))} \right)
\end{align*} and the assertion holds  \hfill $\blacksquare$
\end{proof}


{\section{NUMERICAL ILLUSTRATION}\label{sec:numer}
This section provides a numerical validation of the fully-corrective Frank-Wolfe (fcFW) method from Algorithm~\ref{alg:DFW}. We apply fcFW to the Gaussian deconvolution example introduced in Section~\ref{sec:deconvolution} to evaluate its effectiveness in recovering probability measures from noisy observations.  

Following the setup in~\cite{yan2024learning}, we conduct experiments where the latent variables \( W_i \), for \( i=1,\dots,n \), in~\eqref{eq:deconvolution} are sampled from two distinct distributions:  
\begin{itemize}
    \item \textbf{Discrete distribution:}  
    \begin{equation}\label{mu_a}
        \mu_a = \frac{1}{3} \delta_{-1} + \frac{1}{3} \delta_{1} + \frac{1}{3} \delta_{10}.
    \end{equation}
    Prior studies~\cite{yan2024learning, jin2016local} have shown that classical expectation-maximization (EM) and gradient descent methods struggle in this setting due to poor local optima. We assess whether fcFW provides a robust alternative.  
    \item \textbf{Continuous distribution:}  
    \begin{equation}
        \mu_b = \mathcal{N}(0,I_d).
    \end{equation}
    We conduct experiments with \( d = 10 \) to evaluate the scalability of fcFW in higher dimensions.
\end{itemize}

\begin{figure}[htbp]
    \centering
    \begin{subfigure}[b]{0.45\textwidth}
        \centering
        \includegraphics[width=\textwidth]{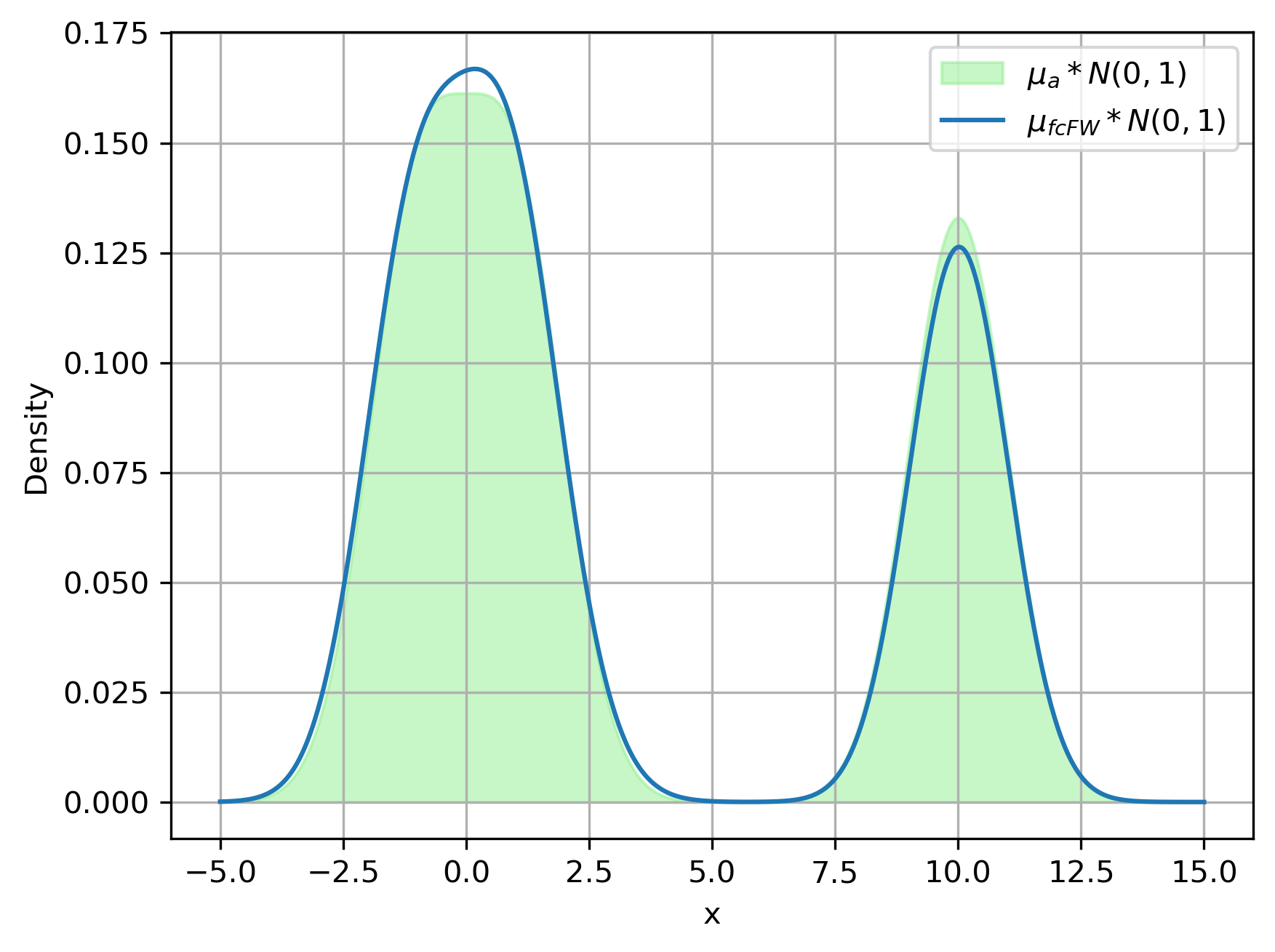}
        \caption{Density comparison}
        \label{fig:ex1density}
    \end{subfigure}
    \hfill
    \begin{subfigure}[b]{0.45\textwidth}
        \centering
        \includegraphics[width=\textwidth]{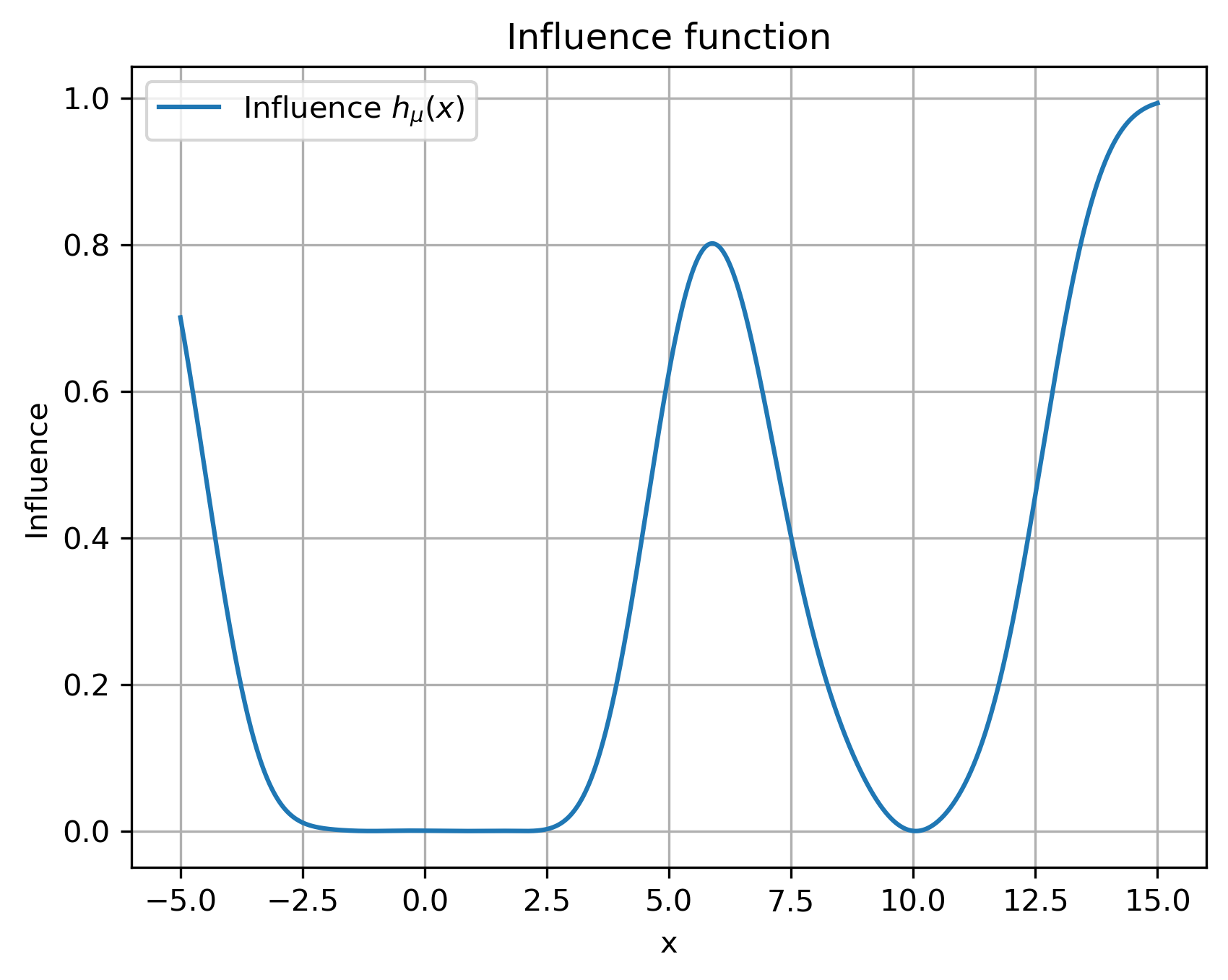}
        \caption{Influence function at $\mu_{fc-FW}$}
        \label{fig:ex1Infl}
    \end{subfigure}
    \caption{fcFW Results for Gaussian Deconvolution (Discrete Case, \(n=1500\)).  
    (a) Comparison of the recovered density \(\mu_{fcFW} * N(0,1)\) (blue) with the population density \(\mu_a * N(0,1)\) (shaded).  
    (b) The influence function at \(\mu_{fcFW}\) is non-negative, verifying global optimality as stated in Lemma~\ref{lem:optCondition}.}

    \label{fig:experiment1}
\end{figure}

In our first experiment, we examine the discrete distribution case. Prior work~\cite{jin2016local} showed that the log-likelihood of a three-component Gaussian mixture model in dimension \( d=1 \) has a poor local maximum, where expectation-maximization (EM) and gradient descent often get trapped. This issue was further confirmed through numerical experiments in~\cite{yan2024learning}. These challenges make this a useful test case for evaluating whether fcFW offers a more reliable alternative.  

To assess the performance of fcFW, we generate \( n = 1500 \) samples \(\{Y_i\}_{1\leq i \leq n}\) from the distribution \(\mu_a * N(0,1)\). Keeping these samples fixed, we apply the fcFW method and obtain the probability measure \(\mu_{fcFW}\) after 200 iterations. We repeat the experiment 20 times independently.  

Figure~\ref{fig:experiment1} shows the results from one trial. In Figure~\ref{fig:ex1density}, the density of \(\mu_{fcFW} * N(0,1)\) closely matches the population density \(\mu_a * N(0,1)\), indicating successful recovery. Figure~\ref{fig:ex1Infl} shows that the influence function at \(\mu_{fcFW}\) remains non-negative, which, by Lemma~\ref{lem:optCondition}, confirms convergence to a global optimum. Similar results were observed across all 20 trials.

Previous studies~\cite{jin2016local, yan2024learning} found that EM and gradient descent struggle in this setting due to poor local optima. While we do not implement these methods here, our results suggest that fcFW reliably recovers the underlying measure, making it a promising alternative.

\begin{figure}[htbp]
    \centering
    \includegraphics[width=0.5\textwidth]{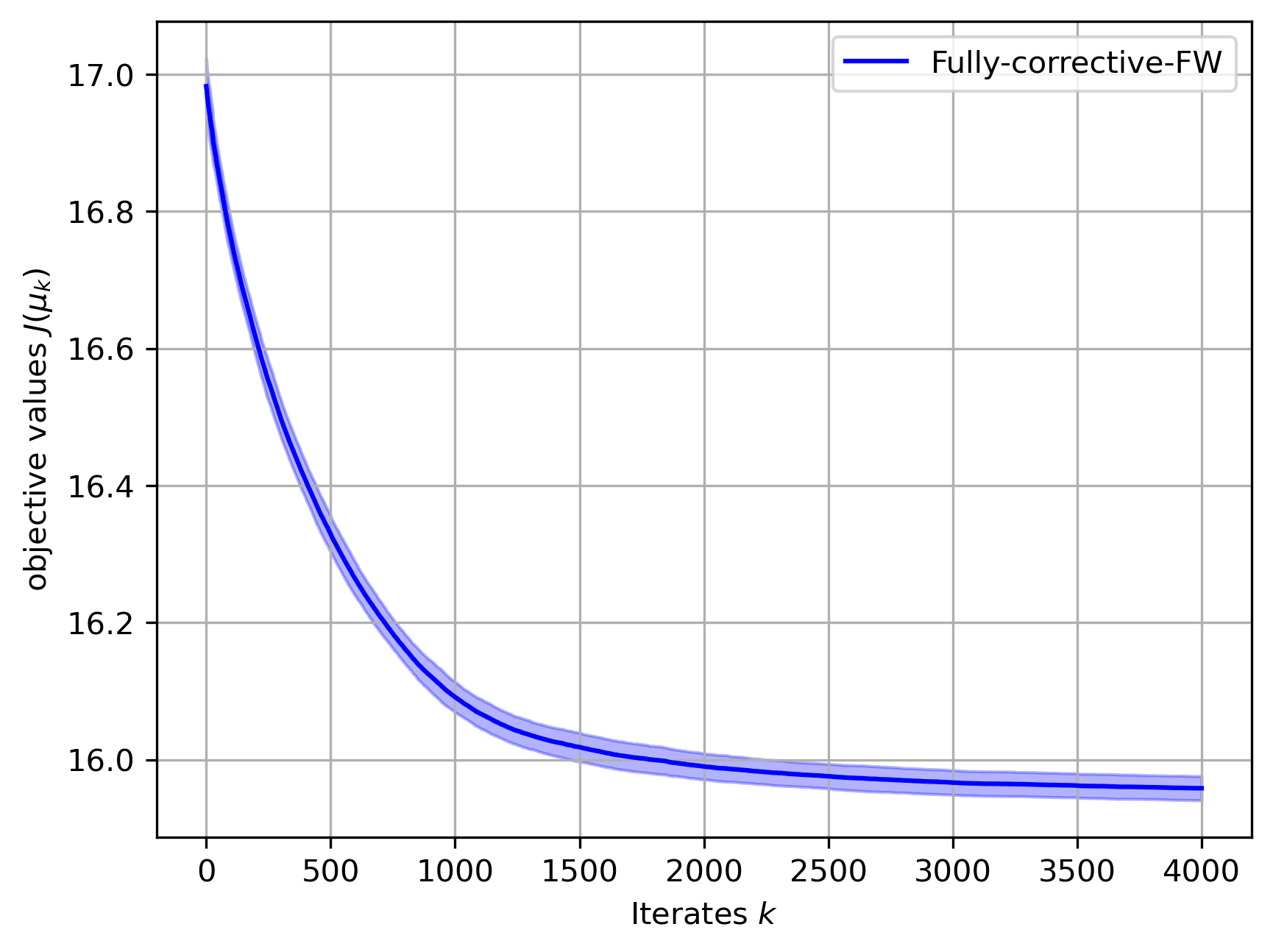}
    \caption{Objective values \( J(\mu_k) \) of fcFW over 4000 iterations for the continuous case \( \mu_b \) with \( d=10 \). The shaded region represents the standard deviation over 10 independent trials.}
    \label{fig:ex2obj}
\end{figure}

In the second experiment, we test the performance of fcFW in a higher-dimensional setting with \( d=10 \). Higher dimensions introduce additional challenges, making it important to assess scalability. We generate \( n = 1500 \) samples \(\{Y_i\}_{1\leq i \leq n}\) from the distribution \(\mu_b * \mathcal{N}(0,I_{10})\) and keep them fixed throughout the experiment. We apply fcFW and obtain the probability measure \(\mu_{\text{fcFW}}\) after 4000 iterations across 10 independent trials.  

Figure~\ref{fig:ex2obj} shows the objective values of fcFW over iterations for these trials, demonstrating a steady decrease and confirming convergence. This also highlights an advantage of fcFW over gridding-based approaches, as discussed in Section~\ref{sec:grid}. Unlike gridding, which requires a predefined discretization of the space and becomes computationally expensive in higher dimensions, fcFW operates directly in the infinite-dimensional space, avoiding the need for discretization.  

Prior work~\cite{yan2024learning} has shown that particle-based methods, such as the Wasserstein-Fisher-Rao (WFR) gradient flow, perform well in both discrete and continuous settings. These methods approximate probability measures with empirical measures and update both the support and weights over time through a flow-based approach. While particle methods are effective, fcFW provides a direct optimization framework in infinite-dimensional space, making it a promising alternative.
}

\section{CONCLUDING REMARKS}\label{sec:conc} Incorporating the influence function as the first variational object within a primal recursion such as Frank-Wolfe provides a powerful first-order recursion for stochastic optimization over probability spaces. The resulting paradigm is especially important since there appear to be important broad contexts such as emergency response and experimental design where the influence function is available as a natural first-order derivative for incorporation into a deterministic or a stochastic oracle. Furthermore, in analogy with stochastic gradient recursions in Euclidean spaces, these recursions exhibit convergence behavior without imposing strict conditions such as CLLF or sparsity. Ongoing work tries to extend the proposed methods to incorporate different types of constraints, e.g., structural constraints such as the existence of an $L_2$ density, or functional constraints such as restrictions on the moments.   

{ An interesting direction for future investigation is understanding the connection between the FW methods proposed here and common particle-based methods such as the Wasserstein-Fisher-Rao (WFR) gradient flow~\cite{yan2024learning}. Specifically, while FW and WFR methods follow different update strategies, they share theoretical connections, as FW can also be viewed from a particle-based perspective. WFR keeps a fixed number of particles and updates both their support and weights at each iteration, whereas FW iteratively adds new particles while keeping the support of existing ones unchanged, adjusting only their weights. Future research could explore a hybrid approach that combines these strategies—allowing for both the addition of new particles and updates to existing supports under different conditions. Moreover, the influence function used in FW methods also plays a key role in WFR and other particle-based approaches, further linking these frameworks.}

\section*{Acknowledgments}
This work was partially supported by National Science Foundation grants CMMI-2035086, DMS-2230023 and OAC-2410950. { We thank the editorial team for helpful reports that improved the content and exposition of the paper.}

\bibliographystyle{plain}
\bibliography{references}

\begin{thebibliography}{10}

\bibitem{2005ambluisav}
Luigi Ambrosio, Nicola Gigli, and Giuseppe Savar{\'e}.
\newblock {\em Gradient flows: in metric spaces and in the space of probability measures}.
\newblock Springer Science \& Business Media, 2005.

\bibitem{2017andgensha}
Isaiah Andrews, Matthew Gentzkow, and Jesse~M Shapiro.
\newblock Measuring the sensitivity of parameter estimates to estimation moments.
\newblock {\em The Quarterly Journal of Economics}, 132(4):1553--1592, 2017.

\bibitem{1976bar}
R.~G. Bartle.
\newblock {\em The Elements of Real Analysis}.
\newblock Wiley, New York, NY, 1976.

\bibitem{1999bil}
Patrick Billingsley.
\newblock {\em Convergence of probability measures}.
\newblock Wiley series in probability and statistics. Probability and statistics. Wiley, New York, 2nd ed. edition, 1999.

\bibitem{2018botcurnoc}
L\'{e}on Bottou, Frank~E. Curtis, and Jorge Nocedal.
\newblock Optimization methods for large-scale machine learning.
\newblock {\em SIAM Review}, 60(2):223--311, 2018.

\bibitem{2017boygeorec}
Nicholas Boyd, Geoffrey Schiebinger, and Benjamin Recht.
\newblock The alternating descent conditional gradient method for sparse inverse problems.
\newblock {\em SIAM Journal on Optimization}, 27(2), 2017.

\bibitem{2013brepik}
K.~Bredies and H.~Pikkarainen.
\newblock Inverse problems in spaces of measures.
\newblock {\em ESAIM. Control, optimisation and calculus of variations}, 19(1):190--218, 2013.

\bibitem{2015bub}
S{\'e}bastien Bubeck et~al.
\newblock Convex optimization: Algorithms and complexity.
\newblock {\em Foundations and Trends{\textregistered} in Machine Learning}, 8(3-4):231--357, 2015.

\bibitem{2013canfer}
Emmanuel~J Cand{\`e}s and Carlos Fernandez-Granda.
\newblock Super-resolution from noisy data.
\newblock {\em Journal of Fourier Analysis and Applications}, 19:1229--1254, 2013.

\bibitem{2014canfer}
Emmanuel~J Cand{\`e}s and Carlos Fernandez-Granda.
\newblock Towards a mathematical theory of super-resolution.
\newblock {\em Communications on Pure and Applied Mathematics}, 67(6):906--956, 2014.

\bibitem{2005canuto}
Claudio Canuto and Karsten Urban.
\newblock Adaptive optimization of convex functionals in {Banach} spaces.
\newblock {\em SIAM Journal on Numerical Analysis}, 42(5):2043--2075, 2005.

\bibitem{2022caretal}
Jos{\'e}~A Carrillo, Katy Craig, Li~Wang, and Chaozhen Wei.
\newblock Primal dual methods for {W}asserstein gradient flows.
\newblock {\em Foundations of Computational Mathematics}, pages 1--55, 2022.

\bibitem{2019caretal}
Jos{\'e}~Antonio Carrillo, Katy Craig, and Francesco~S Patacchini.
\newblock A blob method for diffusion.
\newblock {\em Calculus of Variations and Partial Differential Equations}, 58:1--53, 2019.

\bibitem{2018cheetal}
Victor Chernozhukov, Denis Chetverikov, Mert Demirer, Esther Duflo, Christian Hansen, Whitney Newey, and James Robins.
\newblock Double/debiased machine learning for treatment and structural parameters, 2018.

\bibitem{2022cheetal}
Victor Chernozhukov, Juan~Carlos Escanciano, Hidehiko Ichimura, Whitney~K Newey, and James~M Robins.
\newblock Locally robust semiparametric estimation.
\newblock {\em Econometrica}, 90(4):1501--1535, 2022.

\bibitem{2022achi}
L.~Chizat.
\newblock Convergence rates of gradient methods for convex optimization in the space of measures.
\newblock {\em Open Journal of Mathematical Optimization}, 3:1--19, 2022.

\bibitem{2022bchi}
L.~Chizat.
\newblock Sparse optimization on measures with over-parameterized gradient descent.
\newblock {\em Mathematical programming}, 194(1-2):487--532, 2022.

\bibitem{2019chublagly}
Casey Chu, Jose Blanchet, and Peter Glynn.
\newblock Probability functional descent: A unifying perspective on {GAN}s, variational inference, and reinforcement learning, 2019.

\bibitem{daskin83}
M.~Daskin.
\newblock A maximal expected covering location model: formulation, properties, and heuristic solution.
\newblock {\em Transportation Science}, 17:48--70, 1983.

\bibitem{2012degam}
Yohann De~Castro and Fabrice Gamboa.
\newblock Exact reconstruction using {Beurling} minimal extrapolation.
\newblock {\em Journal of Mathematical Analysis and Applications}, 395(1):336--354, 2012.

\bibitem{1970demrub}
Vladimir~Fedorovich Demynov and Aleksandr~Moiseevich Rubinov.
\newblock {\em Approximate methods in optimization problems}.
\newblock American Elsevier Company, 1970.

\bibitem{1978dunhar}
Joseph~C Dunn and S~Harshbarger.
\newblock Conditional gradient algorithms with open loop step size rules.
\newblock {\em Journal of Mathematical Analysis and Applications}, 62(2):432--444, 1978.

\bibitem{durrett2019probability}
Rick Durrett.
\newblock {\em Probability: theory and examples}, volume~49.
\newblock Cambridge university press, 2019.

\bibitem{eftekhari2019sparse}
Armin Eftekhari and Andrew Thompson.
\newblock Sparse inverse problems over measures: Equivalence of the conditional gradient and exchange methods, 2019.

\bibitem{2012fedhac}
Valerii~V Fedorov and Peter Hackl.
\newblock {\em Model-oriented design of experiments}, volume 125.
\newblock Springer Science \& Business Media, 2012.

\bibitem{2013fer}
Carlos Fernandez-Granda.
\newblock Support detection in super-resolution.
\newblock {\em arXiv preprint arXiv:1302.3921}, 2013.

\bibitem{1983fer}
LT~Fernholz.
\newblock Lecture notes in statistics.
\newblock {\em Von Mises Calculus for Statistical Functionals}, 19, 1983.

\bibitem{2011fer}
Luisa Fernholz.
\newblock Functional derivatives in statistics: Asymptotics and robustness, 2011.

\bibitem{2009firetal}
Sergio Firpo, Nicole~M Fortin, and Thomas Lemieux.
\newblock Unconditional quantile regressions.
\newblock {\em Econometrica}, 77(3):953--973, 2009.

\bibitem{frank1956algorithm}
Marguerite Frank, Philip Wolfe, et~al.
\newblock An algorithm for quadratic programming.
\newblock {\em Naval Research Logistics Quarterly}, 3(1-2):95--110, 1956.

\bibitem{2019futcui}
Futoshi Futami, Zhenghang Cui, Issei Sato, and Masashi Sugiyama.
\newblock {B}ayesian posterior approximation via greedy particle optimization.
\newblock In {\em Proceedings of the AAAI Conference on Artificial Intelligence}, volume~33, pages 3606--3613, 2019.

\bibitem{2019garulb}
Sebastian Garreis and Michael Ulbrich.
\newblock An inexact trust-region algorithm for constrained problems in {Hilbert} space and its application to the adaptive solution of optimal control problems with {PDEs}.
\newblock {\em Preprint, submitted, Technical University of Munich}, 2019.

\bibitem{geer2000empirical}
Sara~A Geer.
\newblock {\em Empirical Processes in M-estimation}, volume~6.
\newblock Cambridge university press, 2000.

\bibitem{1974ham}
Frank~R Hampel.
\newblock The influence curve and its role in robust estimation.
\newblock {\em Journal of the American Statistical Association}, 69(346):383--393, 1974.

\bibitem{2022shaetal}
Shane~G Henderson, Pieter L van~den Berg, Caroline~J Jagtenberg, and Hemeng Li.
\newblock How should volunteers be dispatched to out-of-hospital cardiac arrest cases?
\newblock {\em Queueing Systems}, 100(3-4):437--439, 2022.

\bibitem{1996hub}
Peter~J. Huber.
\newblock {\em Robust statistical procedures}.
\newblock CBMS-NSF regional conference series in applied mathematics ; 68. Society for Industrial and Applied Mathematics SIAM, 3600 Market Street, Floor 6, Philadelphia, PA 19104, Philadelphia, Pa, 2nd ed. edition, 1996.

\bibitem{jin2016local}
Chi Jin, Yuchen Zhang, Sivaraman Balakrishnan, Martin~J Wainwright, and Michael~I Jordan.
\newblock Local maxima in the likelihood of {Gaussian} mixture models: Structural results and algorithmic consequences.
\newblock {\em Advances in Neural Information Processing Systems}, 29, 2016.

\bibitem{kent2021frankwolfe}
Carson Kent, Jose Blanchet, and Peter Glynn.
\newblock {Frank-Wolfe} methods in probability space, 2021.

\bibitem{1974kie}
J.~Kiefer.
\newblock General equivalence theory for optimum designs (approximate theory).
\newblock {\em The Annals of Statistics}, 2(5):849--879, 1974.

\bibitem{lacostejulien2016convergence}
Simon Lacoste-Julien.
\newblock Convergence rate of {Frank-Wolfe} for non-convex objectives, 2016.

\bibitem{2016liudil}
Qiang Liu and Dilin Wang.
\newblock Stein variational gradient descent: A general purpose {Bayesian} inference algorithm.
\newblock {\em Advances in Neural Information Processing Systems}, 29, 2016.

\bibitem{magnus2019matrix}
Jan~R Magnus and Heinz Neudecker.
\newblock {\em Matrix differential calculus with applications in statistics and econometrics}.
\newblock John Wiley \& Sons, 2019.

\bibitem{2018mei}
Song Mei, Andrea Montanari, and Phan-Minh Nguyen.
\newblock A mean field view of the landscape of two-layer neural networks.
\newblock {\em Proceedings of the National Academy of Sciences}, 115(33), jul 2018.

\bibitem{2000molzuy}
I.~Molchanov and S.~Zuyev.
\newblock Tangent sets in the space of measures: With applications to variational analysis.
\newblock {\em Journal of Mathematical Analysis and Applications}, 249(2):539--552, 2000.

\bibitem{2002molzuy}
I.~Molchanov and S.~Zuyev.
\newblock Steepest descent algorithms in a space of measures.
\newblock {\em Statistics and Computing}, 12:115--123, 2002.

\bibitem{2018myaetal}
Aung Myat, Kyoung-Jun Song, and Thomas Rea.
\newblock Out-of-hospital cardiac arrest: current concepts.
\newblock {\em The Lancet}, 391(10124):970--979, 2018.

\bibitem{2004nes}
Y.~Nesterov.
\newblock {\em Introductory Lectures on Convex Optimization A Basic Course}.
\newblock Applied Optimization, 87. Springer US, New York, NY, 1st edition, 2004.

\bibitem{2018nes}
Yurii Nesterov et~al.
\newblock {\em Lectures on convex optimization}, volume 137.
\newblock Springer, 2018.

\bibitem{2000oka}
Atsuyuki Okabe.
\newblock {\em Spatial tessellations: concepts and applications of Voronoi diagrams}.
\newblock Wiley series in probability and statistics. Applied probability and statistics section. Wiley, Chichester, 2nd edition, 2000.

\bibitem{pokutta2023frankwolfe}
Sebastian Pokutta.
\newblock The {Frank-Wolfe} algorithm: a short introduction, 2023.

\bibitem{1983puk}
F.~Pukelsheim.
\newblock {\em Optimal Design of Experiments}.
\newblock Wiley, New York, 1983.

\bibitem{2004raoetal}
Murali Rao, Yunmei Chen, Baba~C Vemuri, and Fei Wang.
\newblock Cumulative residual entropy: a new measure of information.
\newblock {\em IEEE transactions on Information Theory}, 50(6):1220--1228, 2004.

\bibitem{reddi2016stochastic}
Sashank~J. Reddi, Suvrit Sra, Barnabas Poczos, and Alex Smola.
\newblock Stochastic {Frank-Wolfe} methods for nonconvex optimization, 2016.

\bibitem{2015san}
Filippo Santambrogio.
\newblock Optimal transport for applied mathematicians.
\newblock {\em Birk{\"a}user, NY}, 55(58-63):94, 2015.

\bibitem{2015sha}
Xiaofeng Shao.
\newblock Self-normalization for time series: a review of recent developments.
\newblock {\em Journal of the American Statistical Association}, 110(512):1797--1817, 2015.

\bibitem{2009shadenrus}
A.~Shapiro, D.~Dentcheva, and A.~Ruszczynski.
\newblock {\em Lectures on Stochastic Programming: Modeling and Theory}.
\newblock Society for Industrial and Applied Mathematics, Philadelphia, Pennsylvania, 2009.

\bibitem{2017ulbzie}
Stefan Ulbrich and Jan~Carsten Ziems.
\newblock Adaptive multilevel trust-region methods for time-dependent {PDE-}constrained optimization.
\newblock {\em Portugaliae Mathematica}, 74(1):37--67, 2017.

\bibitem{berhenjagli21}
Pieter~L. {van den Berg}, Shane~G. Henderson, Caroline~J. Jagtenberg, and Hemeng Li.
\newblock Modeling the impact of community first responders.
\newblock {\em Management Science}, 71(2):992--1008, 2025.

\bibitem{yan2024learning}
Yuling Yan, Kaizheng Wang, and Philippe Rigollet.
\newblock Learning {Gaussian} mixtures using the {Wasserstein--Fisher--Rao} gradient flow.
\newblock {\em The Annals of Statistics}, 52(4):1774--1795, 2024.

\end{thebibliography}
\end{document}